\documentclass{amsart}
\usepackage{amsmath,amsthm,amsfonts, amssymb,amscd}
\usepackage{amsmath,amsthm,amsfonts,amssymb,amscd,epic,eepic,bm,stmaryrd}
\usepackage[all]{xy}

\newtheorem{theorem}{Theorem}[section]
\newtheorem{lemma}[theorem]{Lemma}%[section]
\newtheorem{definition}[theorem]{Definition}%[section]
\newtheorem{example}[theorem]{Example}%[section]
\newtheorem{remark}[theorem]{Remark}%[section]
\newtheorem{corollary}[theorem]{Corollary}%[section]
\newtheorem{proposition}[theorem]{Proposition}

\def\ld{\mathord{\backslash}}
\def\rd{\mathord{/}}

\def\ld{\mathord{\backslash}}
\def\rd{\mathord{/}}
\newcommand{\bss}{\mathbin{\backslash\mkern-6mu\backslash}}
\begin{document}
\title[Kites and Residuated Lattices]{Kites and Residuated Lattices}
\author[Michal Botur, Anatolij Dvure\v{c}enskij]{Michal Botur$^1$, Anatolij Dvure\v censkij$^{1,2}$}
\date{}%
%\maketitle
\thanks{Both authors gratefully acknowledge  the support by GA\v{C}R 15-15286S. AD thanks also the grants APVV-16-0073 and VEGA No. 2/0069/16 SAV}
\address{$^1$Palack\' y University Olomouc, Faculty of Sciences, t\v r. 17.listopadu 12, CZ-771 46 Olomouc, Czech Republic}
\address{$^2$Mathematical Institute, Slovak Academy of Sciences, \v{S}tef\'anikova 49, SK-814 73 Bratislava, Slovakia}
\email{michal.botur@upol.cz, dvurecen@mat.savba.sk}


\maketitle

\begin{abstract}
We investigate a construction of an integral residuated lattice starting from an integral residuated lattice and two sets with an injective mapping from one set into the second one. The resulting algebra has a shape of a Chinese cascade kite, therefore, we call this algebra simply a kite. We describe subdirectly irreducible kites and we classify them. We show that the variety of integral residuated lattices generated by kites is generated by all finite-dimensional kites. In particular, we describe some homomorphisms among kites.

\end{abstract}
\date{}%
%\maketitle

\section{Introduction}

There are many lattice-ordered structures that are very tightly connected with lattice-ordered groups (= $\ell$-groups). Such situations are observed for example with MV-algebras, an algebraic semantics of the inifinite-valued \L ukasiewicz logic, see \cite{Cha}, when by \cite{Mun}, every MV-algebra is an interval in a unique Abelian $\ell$-group with strong unit and vice versa. Similarly, every  pseudo MV-algebra, a non-commutative generalization of MV-algebras introduced in \cite{GeIo, Rac}, is an interval in a unital $\ell$-group not necessarily Abelian, and vice-versa. Moreover, there is a categorical equivalence of the category of pseudo MV-algebras and the category of unital $\ell$-groups, see \cite{DvuA}. BL-algebras, introduced by H\'ajek \cite{Haj}, are an algebraic semantics of the classical fuzzy logic generalizing MV-algebras, and pseudo BL-algebras are a non-commutative generalization of BL-algebras which were introduced in \cite{DGI1, DGI2}. By \cite{AgMo, DvAgl}, every linearly ordered pseudo BL-algebra can be decomposed into a family of negative cones  and one negative interval of some linearly ordered groups. These algebras give important cases of integral residuated lattices which are connected with $\ell$-groups.

Jipsen and Montagna \cite{JiMo} constructed a subdirectly irreducible pseudo BL-algebra starting from the negative and positive cone of the $\ell$-group $\mathbb Z$ of integers that was not a linearly ordered pseudo BL-algebra  and no BL-algebra. This example was used in \cite{DGK} to show that an open problem from \cite[Problem 3.21]{DGI2}) whether in every pseudo BL-algebra left negation and right negation mutually commute has a negative solution. Because the example resembles a kite with $(\mathbb Z^-)^2$ as a head and $\mathbb Z^+$ as a tail, this examples was said to be a kite. This construction was extended in \cite{DvKo} for an arbitrary $\ell$-group and the resulting algebra is a pseudo BL-algebra, called also a kite pseudo BL-algebra. The basic properties of kites, subdirectly irreducible kites, classification of kites, and situations when a kite gives a pseudo MV-algebra are described in \cite{DvKo} in details.

The aim of the present paper is to give a new type of a construction of an integral residuated lattice starting from an integral residuated lattice, with two sets $I_0$ and $I_1$ satisfying $I_1\subseteq I_0$ and with an injective mapping $\lambda: I_1\to I_0$. The resulting algebra will have a shape of a Chinese cascade kite, therefore it will be called a kite residuated lattice or simply a kite. These new types of integral residuated
lattices enrich theory of residuated lattices and also show a way how residuated lattices can start in particular from $\ell$-groups.

The paper is organized as follows. Basic notions on residuated lattices are presented in Section 2. Section 3 presents a construction of kite residuated lattices. In Section 4 we give some important examples of the construction of kite residuated lattices. Subdirectly irreducible kites are completely described together with classification in Section 5. In particular, we show that a necessary condition to be a kite subdirectly irreducible is that the set $I_0$ is at most countably infinite. We also prove that every kite is a subdirect product of subdirectly irreducible kites. Infinite-dimensional and finite-dimensional kites are described in Section 6, and we show that the variety generated by all kites is generated by all finite-dimensional kites. Finally, Section 7 describes some homomorphisms between two kites.

\section{Basic Notions and Notations}%2

We say that an algebra $\mathbf G=(G;\wedge,\vee,\cdot, \backslash,/,e)$ of type $\langle 2,2,2,2,2,0\rangle$ is a {\it residuated lattice} if $(G;\wedge,\vee)$ is a lattice such that $(G;\cdot, \backslash,/,e)$ is a residuated monoid, i.e. the product (or multiplication) $\cdot$ is associative with unit element $e$, and $x\cdot y \le z$ iff $y\le x\backslash z$ iff $x \le z/y$ for all $x,y,z \in G$.

A residuated  $(G;\wedge,\vee,\cdot, \backslash,/,e)$  is said to be an {\it integral residuated lattice} if the unit element $e$  satisfies $x \le e$ for each $x\in G$.

The operations $\backslash$ and $/$ are called the {\it left residuation} (or the {\it left division}) and the {\it right residuation} (or the {\it right division}), respectively. Multiplications bind stronger than multiplication, which binds stronger than divisions, which in turn binds stronger than the lattice operations $\wedge$ and $\vee$. For more information about residuated lattices see \cite{BlTs, GaTs}.

Now we introduce some equalities
\begin{enumerate}
\item[{\rm (i)}] $x(x\ld y)=x\wedge y= (y \rd x)x$ \quad  ({\it divisibility}),
\item[{\rm (ii)}]  $x\ld y \vee y\ld x = 1 = y\rd x \vee x\rd y$\quad  ({\it prelinearity}),
\item[{\rm (iii)}] $xy = yx$ \quad ({\it commutativity}),

\item[{\rm (iv)}] $x\rd(y\ld x) = x\vee y = (x\rd y)\ld x$.
\end{enumerate}

An integral residuated lattice $\mathbf G$ with a special element $0$ such that $0\le x$ for each $x \in G$ is said to be (1) a {\it pseudo MV-algebra} if identity (iv) holds. A pseudo MV-algebra with commutativity is said to be an {\it MV-algebra}; (2) a {\it pseudo BL-algebra} if divisibility and prelinearity holds in $\mathbf G$. A commutative pseudo BL-algebra is a {\it BL-algebra}. An integrated lattice $\mathbf G$ is (3) a {\it GBL-algebra} if it satisfied divisibility, and a GBL-algebra satisfying prelinearity is said to be a {\it basic pseudo hoop}.

For example, let $\mathbf G=(G;\wedge,\vee,\cdot,^{-1},e)$ be an $\ell$-group and let $G^-:=\{g \in G\colon g \le e\}$ be the negative cone. Then $\mathbf G^-=(G^-\colon \wedge,\vee,\cdot, \ld,\rd,e)$, where $\cdot$ is the group multiplication in $G$, $x\ld y:= (x^{-1}y)\wedge e$, $y\rd x= (yx^{-1})\wedge e$ for $x,y \in G^-$, is an integral residuated lattice. The class $\mathcal{LG}^-$ of negative cones of $\ell$-groups is a variety whose each member is cancellative \cite[Thm 2.12]{GaTs}, and since the group $\mathbb Z$ generates the variety of Abelian $\ell$-groups, the negative cone $\mathbf Z^-=(\mathbb Z^-; \wedge,\vee,\cdot,\ld,\rd,0)$ generates the variety $\mathcal{ALG}^-$ of negative cones of Abelian $\ell$-groups. If $\mathbf G$ is a doubly transitive permutation group, then the variety generated by $\mathbf G^-$ generates the whole variety $\mathcal{LG}^-$, see \cite[Lem 10.3.1]{Gla}.

\section{Kites Residuated Lattices}%3

We present a construction of kite residuated lattices starting from an integral residuated lattice.

Let us have two sets $I_1$ and $I_0$ with $I_1\subseteq I_0$ and an injective mapping $\lambda\colon I_1\longrightarrow I_0$. We define inductively, for each integer $n\ge 1$, the following sets
$$
I_{n+1}=\{i\in I_n\colon \lambda (i)\in I_n\}.
$$
Clearly, if $i\in I_n$ then $\lambda (i)\in I_{n-1}$, and consequently, $\lambda ^{m-n}\colon I_m\longrightarrow I_n$ is a correctly defined mapping (for any $m,n\in\mathbb N$ such that $n\leq m$). As usually, by $\mathbb N$ we denote the set of all integers $n\ge 0$.

Let $\mathbf G=(G;\wedge,\vee,\cdot, \backslash,/,e)$ be an integral residuated lattice. As usually, in residuated monoids, multiplication has higher priority than divisions, and divisions are stronger than lattice connectives $\vee$ and $\wedge$. For any element $x \in G$, we define $x^0=e$ and $x^{n+1}=x^n \cdot x$, $n\ge 0$.

We define a (lexicographic) order on the set
$$
\biguplus_{n\in \mathbb N} G^{I_n}
$$
by $\langle x_i\colon i\in I_n\rangle\leq \langle y_i \colon i\in I_m\rangle$ if and only if  $m<n$ holds or $m=n$ and $x_i\leq y_i$ for all $i\in I_n$. It is clear that $(\biguplus_{n\in \mathbb N} G^{I_n};\leq)$ is a lattice-ordered set. If we denote by $1$ a unique element belonging to $G^{I_0}$  satisfying $1(i)=e$ for any $i\in I_0$, then $1$ is the top element of $\biguplus_{n\in \mathbb N} G^{I_n}$. We notice that it can happen that some $I_n$ is the empty set. Then $G^{I_n}$ is a singleton, we denote it e.g. as $G^{I_n}=\{\langle e\colon i \in I_n \rangle \}$.

Moreover, we define operations $\cdot$, $\sslash$, and $\bss$, product, right division and left division, on the set $\biguplus_{n\in \mathbb N} G^{I_n}$ as follows:
\begin{itemize}
\item[($\cdot$)] For $\langle x_i\colon i\in I_m\rangle, \langle y_i\colon i\in I_n\rangle\in  \biguplus_{n\in \mathbb N} G^{I_n}$, we set
$$\langle x_i\colon i\in I_m\rangle\cdot \langle y_i\colon i\in I_n\rangle=\langle x_{\lambda^n (i)}y_i\colon i\in I_{m+n} \rangle.$$
\item[($\sslash$)]  For $\langle x_i\colon i\in I_m\rangle, \langle y_i\colon i\in I_n\rangle\in  \biguplus_{n\in \mathbb N} G^{I_n}$, we set
$$
\langle y_i\colon i\in I_n\rangle \sslash \langle x_i\colon i\in I_m\rangle =\left\{\begin{array}{lll}
\langle (y\sslash x)_i\colon i\in I_{n-m}\rangle & \mbox{ if } & m\leq n\\
1 &\mbox{ if } & m > n,\\
\end{array}\right.$$
where

$$(y\sslash x)_i=\left\{\begin{array}{lll}
y_{\lambda^{-m}(i)}/x_{\lambda^{-m}(i)} &\mbox{ if }& i\in I_{n-m}\cap\lambda^m (I_n)\\
e &\mbox{ if }& i\in I_{n-m}\setminus \lambda^m (I_n).\end{array}\right.
$$
The injectivity of the mapping $\lambda^m$ guarantees the existence of $\lambda^{-m}$ defined on its domain, and $\lambda^m(I_n)\subseteq I_{n-m}$ if $m\le n$, so that $I_{n-m}\cap\lambda^m (I_n)=\lambda^m (I_n)$.
\item[($\bss$)] For $\langle x_i\colon i\in I_m\rangle, \langle y_i\colon i\in I_n\rangle\in  \biguplus_{n\in \mathbb N} G^{I_n}$, we set

$$
\langle y_i\colon i\in I_n\rangle \bss \langle x_i\colon i\in I_m\rangle =\left\{\begin{array}{lll}
\langle (y\bss x)_i\colon i\in I_{m-n}\rangle & \mbox{ if } & n\leq m\\
1 &\mbox{ if } & n > m,\\
\end{array}\right.$$
where
$$
(y\bss x)_i=\left\{\begin{array}{lll}
y_{\lambda ^{m-n}(i)}\backslash x_i & \mbox{ if } & i\in I_{m-n}\cap I_m\\
e &\mbox{ if } & i\in I_{m-n}\setminus I_m.
\end{array}\right.$$
\end{itemize}

\begin{theorem}\label{th:3.1}
The algebra  $$\textstyle K^\lambda_{I_0,I_1}(\mathbf G):=(\biguplus_{n\in \mathbb N} G^{I_n};\wedge,\vee,\cdot,\bss,\sslash,1)$$ is an integral residuated lattice.
\end{theorem}

\begin{proof}
Having elements $\langle x_i\colon i\in I_m\rangle, \langle y_i\colon i\in I_n\rangle, \langle z_i\colon i\in I_r\rangle\in  \biguplus_{n\in \mathbb N} G^{I_n}$ and using the definition of $\cdot$, we obtain:
\begin{eqnarray*}
& &\langle x_i\colon i\in I_m\rangle\cdot ( \langle y_i\colon i\in I_n\rangle\cdot \langle z_i\colon i\in I_r\rangle ) \\
&=& \langle x_i\colon i\in I_m\rangle\cdot \langle y_{\lambda^r(i)}z_i\colon i\in I_{n+r}\rangle\cdot\\
&=& \langle x_{\lambda^{n+r}(i)}y_{\lambda ^r(i)}z_i\colon i\in I_{m+n+r}\rangle\\
&=& \langle x_{\lambda^{n}(i)}y_{i}\colon i\in I_{m+n}\rangle\cdot \langle z_i\colon i\in I_r\rangle \\
&=& (\langle x_i\colon i\in I_m\rangle\cdot  \langle y_i\colon i\in I_n\rangle )\cdot \langle z_i\colon i\in I_r\rangle.
\end{eqnarray*}

It is easy to prove that $1$ is a neutral element and thus $(\biguplus_{n\in \mathbb N} G^{I_n};\cdot, 1)$ is a monoid. It was mentioned that the above defined order is a lattice-one.

In the last part, we prove the adjointness property. Let us have elements $\langle x_i\colon i\in I_m\rangle, \langle y_i\colon i\in I_n\rangle, \langle z_i\colon i\in I_r\rangle\in  \biguplus_{n\in \mathbb N} G^{I_n}$  such that
$$
\langle x_{\lambda^n(i)}y_i \colon i\in I_{m+n}\rangle =\langle x_i\colon i\in I_m\rangle \cdot \langle y_i\colon i\in I_n\rangle\leq \langle z_i\colon i\in I_r\rangle.
$$
The definition of the lexicographic ordering yields $r\leq m+n$. If $r<m+n$, then
\begin{align*}
\langle y_i\colon i\in I_n\rangle \leq \langle x_i\colon i\in I_m\rangle\bss \langle z_i\colon i\in I_r\rangle&=\begin{cases}\langle (x\bss z)_i\colon i\in I_{r-m}\rangle & \text{ if } m\le r\\
1 & \text{ otherwise }
\end{cases}
\end{align*}
and
\begin{align*}
\langle x_i\colon i\in I_m\rangle \leq \langle z_i\colon i\in I_r\rangle\sslash  \langle y_i\colon i\in I_n\rangle &=\begin{cases}\langle (z\sslash y)_i\colon i\in I_{r-n}\rangle & \text{ if } n\le r\\
1 & \text{ otherwise }
\end{cases}
\end{align*}
holds (because $r-m<n$ and $r-n<m$).

If $r=m+n$, we have $y_i\leq e=(x\bss z)_i$ for all $i\in I_n\setminus I_r$, and $x_{\lambda^n(i)}y_i\leq z_i$ gives us $y_i\leq x_{\lambda^n(i)}\backslash z_i=(x\bss z)_i$ for all $i\in I_{r}$. Analogously, $x_i\leq e= (z\sslash y)_i$ for all $i\in I_m\setminus \lambda^m (I_r)$ holds. If $i\in \lambda^m(I_r)$, then $\lambda^{-m}(i)\in I_r$ and
$$x_{i}y_{\lambda^{-n}(i)}=x_{\lambda^n\lambda^{-n}(i)}y_{\lambda^{-n}(i)}\le z_{\lambda^{-n}(i)},$$
and also
$$x_{i}\le z_{\lambda^{-n}(i)}\rd y_{\lambda^{-n}(i)}=(z\sslash y)_i.$$ Together we have established that
$$
\langle x_i\colon i\in I_m\rangle \cdot \langle y_i\colon i\in I_n\rangle\leq \langle z_i\colon i\in I_r\rangle
$$
implies
$$\langle y_i\colon i\in I_n\rangle \leq \langle x_i\colon i\in I_m\rangle\bss \langle z_i\colon i\in I_r\rangle
$$
and
$$\langle x_i\colon i\in I_m\rangle \leq \langle z_i\colon i\in I_r\rangle\sslash  \langle y_i\colon i\in I_n\rangle
$$
and vice-versa.
\end{proof}

The shape of the algebra $K^\lambda_{I_0,I_1}(\mathbf G):=(\biguplus_{n\in \mathbb N} G^{I_n};\wedge,\vee,\cdot,\bss,\sslash,1)$ resembles a Chinese cascade kite (especially when some $I_n$ is the empty set (consequently, so are all $I_m$ for $m\ge n$). Therefore, we call $K^\lambda_{I_0,I_1}(\mathbf G)$ a {\it kite residuated lattice}, or simply a {\it kite}. Another form of a kite pseudo BL-algebra was defined in \cite{DvKo}, where powers of the positive and negative cone of an $\ell$-group with two injective mappings were used, and the resulting algebra was a pseudo BL-algebra.

\begin{proposition}\label{pr:3.2}
A kite residuated lattice $K^\lambda_{I_0,I_1}(\mathbf G)$ with $I_0\ne \emptyset$ satisfies  prelinearity if and only if prelinearity holds for $\mathbf G$.
\end{proposition}

\begin{proof}
Let prelinearity hold for $\mathbf G$, i.e. $(x\ld y) \vee (y\ld x)=1= (x\rd y) \vee (y\rd x)$, $x,y \in G$. So take $\langle x_i\colon i \in I_m\rangle$ and $\langle y_i\colon i \in I_n\rangle$ from $\biguplus_\mathbb N G^{I_n}$. If $n<m$, then $\langle x_i\colon i \in I_m\rangle \le \langle y_i\colon i \in I_n\rangle$ and $\langle x_i\colon i \in I_m\rangle \bss \langle y_i\colon i \in I_n\rangle = 1$ so that  $(\langle x_i\colon i \in I_m\rangle \bss \langle y_i\colon i \in I_n\rangle )\vee  (y_i\colon i \in I_n\rangle \bss  \langle x_i\colon i \in I_m\rangle)=1$. The same is true if $m<n$.

Let $m=n$. Then $\langle x_i\colon i \in I_m\rangle \bss \langle y_i\colon i \in I_n\rangle = \langle (x\bss y)_i \colon i \in I_0\rangle$, where $(x\bss y)_i = x_i\ld y_i$ if $i \in I_m$, otherwise it is equal $e$. Similarly $\langle y_i\colon i \in I_n\rangle \bss  \langle x_i\colon i \in I_m\rangle =\langle (y\bss x)_i \colon i \in I_0\rangle$, where $(y\bss x)_i = y_i\ld x_i$ if $i \in I_m$, otherwise it is $e$. Since prelinearity holds in $\mathbf G$, we see that the first prelinearity condition holds in $K^\lambda_{I_0,I_1}(\mathbf G)$.

In the same way we establish the second prelinearity condition.

Now let prelinearity hold in $K^\lambda_{I_0,I_1}(\mathbf G)$. Take $x,y \in G$ and let $x_1=\langle x_i\colon i \in I_0\rangle $ and $y_1=\langle y_i\colon i \in I_0 \rangle$ be defined as follows: $x_i=x$ and $y_i=y$ for each $i \in I_0$. Then $(x\bss y)_i=x\ld y$ and $(y\bss x)_i = y\ld x$ for each $i \in I_0$, and prelinearity in the kite $K^\lambda_{I_0,I_1}(\mathbf G)$ implies $(x\ld y)\vee (y\ld x)=e$ in $\mathbf G$. Similarly,  $(x\sslash y)_i= x\rd y$ and $(y\sslash x)_i= y\rd x$ if $i\in I_0$,
which establishes the second prelinearity condition for $\mathbf G$.
\end{proof}

It is worthy of recalling that if $I_0$ is non-empty, then any identity holding in the kite residuated lattice $K^\lambda_{I_0,I_1}(\mathbf G)$ holds also in $\mathbf G$. Indeed, the residuated lattice $\mathbf G^{I_0}$ is a subalgebra of the kite, and $\mathbf G^{I_0}$ can be homomorphically mapped onto $\mathbf G$.

We note that the divisibility equality $x(x\ld y)=x\wedge y= (y \rd x)x$ does not hold, in general, even if it holds in $\mathbf G$. %If $\lambda=Id_{I_1}$ is the identity function on $I_1$, then the mentioned identity holds in the kite residuated lattice iff so holds in $\mathbf G$.

\section{Examples of Kite Residuated Lattices}%4

We present some important examples of kite residuated lattices.

%\begin{example}\label{ex:3.3}

\subsection{Example 1}\label{sub:1}
%{\rm (1)}
Let $I_0=I_1=\emptyset$ and $\mathbf G$ be an integral residuated lattice. Then $G^{I_0}$ and $G^{I_1}$ are singletons and $\lambda:I_1\to I_0$ can be only the empty function, in particular, $\lambda$ is injective. Hence, $I_n=\emptyset$ for each $n\ge 0$. If $x \in G^{I_n}$, then we can represent it as $x=\langle e\colon i \in I_n\rangle $ for each $n\ge 0$. Then $K^\emptyset_{\emptyset,\emptyset}(\mathbf G)$ is isomorphic to the commutative integral residuated lattice $\mathbf{Z}^-=(\mathbb Z^-\colon \wedge, \vee,+,\ld,\rd,0)$, the negative cone of the group of integers, which is subdirectly irreducible. The isomorphism is given by $\langle e\colon i\in I_n\rangle \mapsto -n$, $n \ge 0$. The same is true if $G=\{e\}$. Consequently, the kite is linearly ordered, commutative and subdirectly irreducible.

In addition, the variety $\mathsf{V}(K^\emptyset_{\emptyset, \emptyset}(\mathbf G))$ of integral residuated lattices generated by the kite $K^\emptyset_{\emptyset,\emptyset}(\mathbf G)$ is the variety $\mathcal{ALG}^-$ of the negative cones of Abelian $\ell$-groups. This is true also if $\mathbf G$ is a negative cone of a doubly transitive permutation $\ell$-group, nevertheless that this $\ell$-group generates the variety $\mathcal{LG}$ of $\ell$-groups, see \cite[Lem 10.3.1]{Gla}.

\subsection{Example 2}\label{sub:2}
We can define an antilexicographic product of $\mathbf G$ with $\mathbf Z^-$, written as $\mathbf G \,\overleftarrow{\times}\, \mathbf Z^-$ as follows. The universe of $\mathbf G \,\overleftarrow{\times}\, \mathbf Z^-$ is the direct product $G\times \mathbb Z^-$ ordered with the antilexicographic product and endowed with the product such $(x,-m)\cdot (y,- n)=(x\cdot y, m+n)$, $x,y \in G$, $m,n \in \mathbb N$ and  and with left and right divisions $\ld$ and $\rd$ such that $ (x,-m)\ld (y,-n) = (x\ld y,m-n)$ if $m\le n$ otherwise $(x,-m)\ld (y,-n) =(e,0)=:1$, and $(y,-n)\ld (x,-m) = (y\ld x,m-n)$ if $m\le n$ otherwise $(x,-m)\ld (y,-n) =(e,0)$. Then $\mathbf G \,\overleftarrow{\times}\, \mathbf Z^-$ is an integral residuated lattice.

If $I_0=I_1=\{0\}$, then $\lambda$ is the identity on $I_1$, and $I_n=\{0\}$ for each $n\ge 0$. Then  $K^{Id}_{\{0\},\{0\}}(\mathbf G)\cong   \mathbf G \,\overleftarrow{\times}\, \mathbf Z^-$ under the isomorphism $\langle x\colon i \in I_n\rangle \mapsto (x,-n)$, $x \in G$, $n\ge 0$. This kite is subdirectly irreducible iff $\mathbf G$ is subdirectly irreducible, see the criterion \ref{th:5.4} below.

\subsection{Example 3}\label{sub:3}
Let $I_0=\{0\}$ and $I_1=\emptyset$. The only function from $I_1$ to $I_0$ is the empty function (whence an injection). In addition, $I_m=\emptyset$ for $m\ge 2$. Therefore, $G^{I_0}=G$, $G^{I_n}$ are singletons for each $n\ge 1$. This situation gives the kite $K^\emptyset_{\{0\},\emptyset}(\mathbf G)$ which has the head and a long thin tail. In other words, this kite is an ordinal sum of the $\mathbf G$ on the top and an infinite sequence of two-element Boolean algebras. This kite is subdirectly irreducible iff so is $\mathbf G$.

If $\mathbf G$ is a GBL-algebra, i.e. an integral residuated lattice satisfying  divisibility, then so is the kite $K^\emptyset_{\{0\},\emptyset}(\mathbf G)$. If $\mathbf G$ satisfies the prelinearity, by Proposition \ref{pr:3.2}, then the kite $K^\emptyset_{\{0\}, \emptyset}(\mathbf G)$ satisfies prelinearity, too. If $\mathbf G$ is a basic pseudo hoop, then the kite $K^\emptyset_{\{0\},\{0\}}(\mathbf G)$ is also a basic pseudo hoop.

We recall that according to \cite[Cor 4.2]{DvAgl}, the kite $K^\emptyset_{\{0\},\emptyset}(\mathbf G)$ is a linearly ordered pseudo hoop iff $\mathbf G$ is the negative cone of some linearly ordered group $\mathbf G$.

\subsection{Example 4}\label{sub:4}
Let $I_0\ne \emptyset $ and $I_1=\emptyset$. The only function from $I_1$ to $I_0$ is the empty function (whence an injection). In addition, $I_m=\emptyset$ for $m\ge 2$, and on the top of the kite we have $G^{I_0}$ which is not a singleton if $\mathbf G$ is not trivial, and with an infinite tail consisting of an infinite sequence of singletons. This case can be reduced to the previous example if we change $\mathbf G$ to $\mathbf G^{I_0}$ and an arbitrary non-empty $I_0$ to a singleton.

\subsection{Example 5}\label{sub:5}
If $\mathbf G$ is trivial, i.e. $G=\{e\}$, then $G^{I_n}$ is a singleton for each $n\ge 0$ and whence, $K^\lambda_{I_0,I_1}(\mathbf G)\cong \mathbf Z^-$. Then both $\mathbf G$ and $K^\lambda_{I_0,I_1}(\mathbf G)$ are subdirectly irreducible.

In Theorem \ref{th:5.8} below we will describe all subdirectly irreducible kites with $I_0$ finite and Theorem \ref{th:5.11} will describe all subdirectly irreducible kites with infinite $I_0$ (and hence, countably infinite as we show further).

\section{Subdirectly Irreducible Kites}%5

In what follows, we will characterize subdirectly irreducible kites. We show that every subdirectly kite has $I_0$ at most infinitely countable. In addition, we present a complete classification of subdirectly irreducible kites and we show that every kite is a subdirect product of subdirectly irreducible kites.

Let $\mathbf G$ be an integral residuated lattice. A {\it left conjugate} of an element $x \in G$ by an element $y \in G$ is the element $\lambda_y(x):=y\ld xy$, and its {\it right conjugate} is the element $\rho_y(x):=yx\rd y$. We denote by $\Gamma$ the set of all right and left conjugations in $\mathbf G$.

We say that a subset $F \subseteq G$ of an integral residuated lattice $\mathbf G$ is a {\it filter} if (i) it contains the top element of $\mathbf G$, (ii) if $x,y \in F$, then $xy\in F$, and (iii) if $x\in F$, $y \in G$ and $x\le y$, then $y \in F$. A filter $F$ is {\it normal} if it is closed under both conjugates, i.e. for all $x \in F$ and all $y \in G$, both $y\ld xy, yx\rd y$ belong to $F$. We note that congruences on $\mathbf G$ are in a one-to-correspondence with normal filters, see e.g. \cite[Thm 4.12]{BlTs}: If $F$ is a normal filter, then $\sim_F$ defined by $x\sim_F y$ iff $x\rd y \in F$  and $y\rd x \in F$ (iff $x\ld y \in F$ and $y\ld x \in F$) is a congruence, and conversely, if $\sim$ is a congruence, then $F_\sim :=\{x \in G\colon x\sim e\}$ is a normal filter of $\mathbf G$. In addition, if $F$ is a normal filter of $\mathbf G$, then the quotient $\mathbf G/F$ is an integral residuated lattice.

We note that according to \cite[Lem 5.3]{BlTs}, if $x$ is an element of $\mathbf G$, then the normal ideal $F(x)$ of $\mathbf G$ generated by $x$ is the set
$$F(x)=\{y \in G\colon \gamma_1(x) \cdots \gamma_n(x) \le y, \gamma_i \in \Gamma, i=1,\ldots, n,\ n\ge 1\}.
\eqno(5.1)
$$

\begin{proposition}\label{pr:5.1}
Let $K^\lambda_{I_0,I_1}(\mathbf G)$ be a kite residuation lattice corresponding to an integral residuated lattice $\mathbf G$. Then $G^{I_0}$ is a maximal normal filter of $K^\lambda_{I_0,I_1}(\mathbf G)$.
\end{proposition}

\begin{proof}
It is straightforward to verify that $F=G^{I_0}$ is a filter. If $x=\langle x_i\colon i \in I_k\rangle \in F$ for some $k\ge 1$, let $F_x$ be the filter of $K^\lambda_{I_0,I_1}(\mathbf G)$ generated by $F\cup\{x\}$.

Then $x\cdot x =\langle x_{\lambda^m(i)}x_i\colon i \in I_{2k}\rangle \in F_x$. Repeating this, we see that $x^n:= (x^{n-1}\cdot x) \in G^{I_{nk}}$ and $x^n \in F_x$ for each integer $n \ge 1$. Hence, every $G^{I_n}$ belongs to $F_x$ and $F_x= K^\lambda_{I_0,I_1}(\mathbf G)$ proving $F$ is maximal.

We show that $G^{I_0}$ is normal. So let $x \in F$ and $y \in G^{I_n}$ for some $n\ge 0$.  Then it is easy to see that  both $y\ld xy$ and $yx\rd y$ belong to $F$. There is another way how to prove the normality of $F$: the mapping $\phi: K^\lambda_{I_0,I_1}(\mathbf G) \to \mathbf Z^-$ defined by $\phi(G^{I_n})=-n$, is a homomorphism of residuated lattices, and $F$ is the kernel of $\phi$, so that it is normal.
\end{proof}

Let $K^\lambda_{I_0,I_1}(\mathbf G)$ be a kite residuation lattice. An element $x =\langle x_i\colon i \in I_n\rangle $, where $n\ge 0$, is said to be $\alpha$-{\it dimensional} for some cardinal $\alpha$, if $|\{i \in I_n\colon x_i\ne e\}|=\alpha$. In particular we have one-dimensional elements as well as a finite-dimensional element $x$ if $\alpha =1$ and $\alpha$ is a finite cardinal, respectively.

\begin{proposition}\label{pr:5.2}
Let $F$ be a normal filter of an integral residuated lattice $\mathbf G$. We denote by $F^{I_0}$ the set
$$ F^{I_0}:= \{\langle x_i\colon i \in I_0\rangle\colon \text{ where } x_i\in N \text{ for all } i \in I_0\}
$$
and let $F^{I_0}_f$ be the system of finite-dimensional elements of $F^{I_0}$. Then $F^{I_0}$ and $F^{I_0}_f$ are normal filters of the kite residuation lattice $K^\lambda_{I_0,I_1}(\mathbf G)$.

Conversely, let $F$ be a proper normal filter of $K^\lambda_{I_0,I_1}(\mathbf G)$. Given $k \in I_0$, let $\pi_k(\langle x_i\colon i \in I_0\rangle)= x_k$. Then $\pi_k(F)=\{\pi_k(x)\colon x \in F\}$ is a normal filter of $\mathbf G$.
\end{proposition}

\begin{proof}
The proof of the first statement follows the same steps as the proof of Proposition \ref{pr:5.1}.

The second part: Since $F$ is a proper filter of the kite, we have $F \subseteq G^{i_0}$. Consequently, $\pi_k(F)$ is a normal filter of $\mathbf G$.
\end{proof}

\begin{proposition}\label{pr:5.3}
If $K^\lambda_{I_0,I_1}(\mathbf G)$ with $I_0\ne \emptyset$ is a subdirectly irreducible kite residuation lattice corresponding to an integral residuated lattice $\mathbf G$, then $\mathbf G$ is a subdirectly irreducible residuation lattice.
\end{proposition}

\begin{proof}
If $\mathbf G=\{e\}$, the statement is satisfied trivially. So let $\mathbf G$ be non-trivial and assume the opposite, i.e. $\mathbf G$ is not subdirectly irreducible. Then there is a set $\{N_s\colon s \in S\}$ of non-trivial normal filters of $\mathbf G$ such that $\bigcap_{s \in S} N_s=\{e\}$. By Proposition \ref{pr:5.2}, every $N_s^{I_0}$ is a normal filter of the kite $K^\lambda_{I_0,I_1}(\mathbf G)$. Let $x=\langle x_i\colon i \in I_0 \rangle \in \bigcap_{s\in S} N_s^{I_0}$. Then $\pi_k(x) \in N_s$ for each coordinate $k \in I_0$ for each $s \in S$. Hence, $\pi_k(x)=\{e\}$ and $x=1$ which shows that the system of normal filters $\{N_s^{I_0}\colon s \in S\}$ intersects trivially. Thus, the kite $K^\lambda_{I_0,I_1}(\mathbf G)$ is not subdirectly irreducible.
\end{proof}

\begin{theorem}\label{th:5.4}
Let $K^\lambda_{I_0,I_1}(\mathbf G)$ with $I_1$ non-empty be a kite residuation lattice corresponding to a non-trivial integral residuated lattice $\mathbf G$. The following statements are equivalent:

\begin{enumerate}
\item[{\rm (1)}] $\mathbf G$ is subdirectly irreducible and for all $i,j \in I_0$, there is an integer $m\ge 0$ such that $\lambda^m(i)=j$ or $\lambda^m (j)=i$.
\item[{\rm (2)}] $K^\lambda_{I_0,I_1}(\mathbf G)$ is subdirectly irreducible.
\end{enumerate}
\end{theorem}

\begin{proof}
(1) $\Rightarrow$ (2). Let $N$ be the least non-trivial normal filter of $\mathbf G$. According to Proposition \ref{pr:5.2}, the set $N^{I_0}_f$ is a normal filter of the kite $K^\lambda_{I_0,I_1}(\mathbf G)$. In what follows, we show that $N^{I_0}_f$ is the least normal filter of the kite. We note that for any element $x \in N^{I_0}_f\setminus \{1\}$, there is a one-dimensional element $x' \in N^{I_0}_f$ such that $x \le x'< 1$. Therefore, to prove that $N^{I_0}_f$ is the least normal filter of the kite, it is sufficient to show that any one-dimensional element $x \in N^{I_0}_f\setminus \{1\}$ generates $N^{I_0}_f$. Without loss of generality, assume $x =\langle x_0,e,\ldots \rangle$ where $x_0 \ne e$; this is possible in view of a suitable reordering of $I_0$ regardless of its cardinality. Since $N$ is the least non-trivial filter of $\mathbf G$, the element $x_0$ generates $N$. We claim that the element $x$ generates all one-dimensional elements of $N^{I_0}_f$ of the form $\langle y_0,e,e,\ldots \rangle$. Choose an index $i \in I_0$. By the assumptions, there is an integer $m\ge 0$ such that $\lambda^m(0)=i$ or $\lambda^m(i)=0$.
Using (5.1), we have for the left and right conjugations $\lambda_y^m(x)$ and $\rho_y^m(x)$, where $x=\langle x_0,e,e,\ldots \rangle$ and $y=\langle y_0,e,e,\ldots \rangle$, the following cases:

\begin{itemize}

\item if $\lambda^m(0)=i$, then $\lambda_y^m(x)=\langle e,\ldots, e, \lambda^m_{y_0}(x_0), e,\ldots \rangle$,

\item if $\lambda^m(i)=0$, then $\rho^m_y(x)=\langle e,\ldots, e, \rho^m_{y_0}(x_0),e,\ldots \rangle$.
\end{itemize}
Re-numbering $I_0$ if necessary, we may assume that the elements $\lambda^m_{y_0}(x_0)$ and $\rho^m_{y_0}(x_0)$ occur at the $m$-th co-ordinate. Therefore, the element $x=\langle x_0, e,e,\ldots \rangle$ generates the normal filter $N^{I_0}_f$.

(2) $\Rightarrow$ (1). Let the kite $K^\lambda_{I_0,I_1}(\mathbf G)$  be subdirectly irreducible. By Proposition \ref{pr:5.3}, we can assume that  $\mathbf G$ is subdirectly irreducible, and let (1) fail. Then there are two indexes $i,j \in I_0$ such that $\lambda^m(i)\ne j$ and $\lambda^m(j)\ne i$ for each integer $m \ge 0$. Similarly as in the proof of \cite[Thm 5.5]{DvKo}, we say that such $i$ and $j$ are {\it disconnected}; otherwise, $i$ and $j$ are {\it connected}. Let $K$ be a maximal subset of $I_0$ such that all elements of $K$ are connected, we called it a {\it connected component} of $I_0$. Then $I_0$ can be decompose into a system of mutually disjoint connected components of $I_0$. Let $K_1$ and $K_2$ be two different connected components of $I_0$. Let $N^{K_1}$ be the system of all elements $\langle x_i\colon i \in I_0\rangle$ such that if $x_i\ne e$, then $i \in K_2$. In the same way we define $N^{K_2}$. Then $N^{K_1}$ and $N^{K_2}$ are filters of the kite. Applying the left and right conjugations to $N^{K_1}$ and $N^{K_2}$ we have that both filters are also normal. Since $K_1$ and $K_2$ are disjoint, $N^{K_1}\cap N^{K_2}=\{1\}$, which contradicts the assumption that the kite is subdirectly irreducible. Therefore, all indexes $i$ and $j$ of $I_0$ are connected, which completes the proof.
\end{proof}

We note that if (1) of the latter theorem holds, $|I_0|>1$, and $I_1$ is non-empty, then for each $i \in I_1$, $\lambda(i)\ne i$.

In addition, if $\mathbf G$ is trivial (consequently $\mathbf G$ is subdirectly irreducible), then $K^\lambda_{I_0,I_1}(\mathbf G)$ is isomorphic to $\mathbf Z^-$  which is also subdirectly irreducible.

In what follows, we show that if the kite  $K^\lambda_{I_0,I_1}(\mathbf G)$ is subdirectly irreducible, then $I_0$ is at most countable and $\lambda$ is bijective.

\begin{proposition}\label{pr:5.5}
Let $K^\lambda_{I_0,I_1}(\mathbf G)$ be a subdirectly irreducible kite and $\mathbf G$ a non-trivial integral residuated lattice. Then $I_0=I_1\cup \lambda(I_1)$ and $I_0$ is at most countably infinite. %and $\lambda $ is a bijection.
\end{proposition}

\begin{proof}
If $I_0$ is empty, the statement is trivially satisfied. Thus, let $I_0$ be non-void. First, observe that if $I_0\setminus (I_1\cup \lambda(I_1))$ is
non-empty, then any $j\in I_1\cup \lambda(I_1)$  is
disconnected from any $i\in I_0\setminus (I_1 \cup\lambda(I_1))$.
Therefore, $I_0 =I_1\cup \lambda(I_1)$. It follows that $I_0$ is
countable if{}f $I_1$ is. Suppose $I_0$ and $I_1$ are uncountable and pick an $i\in I_0$. Consider the set $P(i) = \{\lambda^m(i)\colon \text{ such that } \lambda^m(i) \text { is defined},  m \in \mathbb Z\}$. Clearly $P(i)$ is
at most countable; so there is a $j\in I_0\setminus P(i)$.
But $P(i)$ exhausts all finite paths of back-and-forth beginning from $i$. Then, $i$ and $j$ are disconnected, contradicting Theorem~\ref{th:5.4}.
\end{proof}

\begin{remark}\label{re:5.6} Under the conditions of Proposition {\rm \ref{pr:5.5}}, $I_n= I_{n+1}\cup \lambda(I_{n+1})$ for $n \ge 1$. This can be proved in the same way as the equality $I_0=I_1\cup \lambda(I_1)$ was proved in the foregoing statement. In particular, if $|I_0|=\aleph_0$, then $|I_n|=\aleph_0$ for each $n\ge 1$.
\end{remark}

We note that it can happen that, for a subdirectly irreducible kite $K^\lambda_{I_0,I_1}(\mathbf G)$, $\lambda$ is not necessarily bijective:

\begin{example}\label{ex:5.7}
Let $\mathbf G$ be a subdirectly non-trivial integral residuated lattice.

{\rm (1)} Let $I_0=\{0,1,2,3\}$, $I_1=\{0,1,2\}$, $\lambda: 0\mapsto 1\mapsto 2\mapsto 3$. Then $I_2=\{1,2\}$, $I_3=\{2\}$, $I_m =\emptyset$ for each $m\ge 4$, $I_0$ is a unique connected component of $I_0$, and $K^\lambda_{I_0,I_1}(\mathbf G)$ is a subdirectly irreducible kite. Clearly, $\lambda$ is not bijective.

{\rm (2)} Let $I_0=\{0,1,2,3\}$, $I_1=\{0,1,2,3\}$ and let $\lambda$ be the identity on $I_1$. Then $I_m=I_1$ for each $m\ge 2$ and $K^\lambda_{I_0,I_1}(\mathbf G)$ is not subdirectly irreducible.

{\rm (3)} Let $I_0=\{0,1,2\}=I_1$ and let $\lambda$ be the identity on $I_1$. Then $I_m=I_0$ for each $m\ge 1$ and and $K^\lambda_{I_0,I_1}(\mathbf G)$ is not subdirectly irreducible.
\end{example}

If $I_0$ is a finite set, the kite $K^\lambda_{I_0,I_1}(\mathbf G)$ is said to be {\it finite-dimensional}.

Now we present the following complete descriptions of subdirectly irreducible finite-dimensional kites. In such a case, if $I_0=\{0,\ldots,m-1\}$ and $I_1=\{0,\ldots, n-1\}$, $n,m\ge 1$, we will write $K^\lambda_{m,n}(\mathbf G)$ instead of $K^\lambda_{I_0,I_1}(\mathbf G)$.

As from Example \ref{sub:5} it follows that if $\mathbf G$ is trivial, then the kite $K^\lambda_{I_0,I_1}(\mathbf G)$ is isomorphic to $\mathbf Z^-$ which is subdirectly irreducible. For non-trivial $\mathbf G$, we have the following characterizations of subdirectly irreducible kites.

\begin{theorem}\label{th:5.8}
Let $\mathbf G$ be a non-trivial integral residuated lattice, $K^\lambda_{I_0,I_1}(\mathbf G)$ a subdirectly irreducible kite, and $I_0$ be finite. Then $K^\lambda_{I_0,I_1}(\mathbf G)$ is isomorphic to one of the following kites:
\begin{itemize}
\item[{\rm(1)}] {\rm (i)} $K^\emptyset_{\emptyset,\emptyset}(\mathbf G)\cong \mathbf Z^-$, {\rm (ii)}  $K^{Id}_{\{0\},\{0\}}(\mathbf G)\cong \mathbf G \,\overleftarrow{\times}\, \mathbf Z^-$ and $\mathbf G$ subdirectly irreducible,  {\rm (iii)} $K^\emptyset_{\{0\},\emptyset}(\mathbf G)$ and $\mathbf G$ subdirectly irreducible.
\item[{\rm(2)}] $K^\lambda_{n,n}(\mathbf G)$ with $\lambda(i)=i+1 (\text{mod }n)$ for $i=0,\ldots,n-1$ and $n\ge 2$.
\item[{\rm(3)}] $K^\lambda_{n+1,n}(\mathbf G)$ with $\lambda(i)=i+1$ for $i=0,\ldots,n-1$ and $n\ge 1$.
\end{itemize}
\end{theorem}

\begin{proof}
Assume that $|I_0|=n=|I_1|$. If $n=0,1$, then $K^\lambda_{I_0,I_1}(\mathbf G)$ is isomorphic to corresponding  Examples \ref{sub:1}--\ref{sub:3}.

Now let $n>1$. Then $\lambda$ is a bijection on the set $I_1=\{0,1,\ldots, n-1\}=I_0$. We assert that $\lambda$ is cyclic. If not, then there are $i,j \in I_0$ such that $j$ does not belong to the orbit $P(i) = \{\lambda^m(i)\colon \text{ such that } \lambda^m(i) \text { is defined},  m \in \mathbb Z\}$ of the element $i$, consequently, $i$ and $j$ are disconnected which contradicts Theorem \ref{th:5.4}.  We can renumber $I_1$ following the cycle $\lambda$, so that $\lambda(j)= j+1 (\mbox{mod } n)$, $j \in I_1$.

Now assume that $n=|I_1|<|I_0| =n+m$.  Suppose $m>1$. Then we can find two distinct elements $k_1,k_2 \in I_0\setminus I_1$. An easy inspection shows that $k_1$ and $k_2$ are disconnected, which contradicts Theorem \ref{th:5.4}, and it yields $m=1$.

Hence, if $n=0$, then $n+m=1$ and the kite $K^\lambda_{I_0,I_1}(\mathbf G)$ is isomorphic to the kite $K^\emptyset_{\{0\},\emptyset}(\mathbf G)$ described in Example \ref{sub:3}.

Assume $I_0=\{0,1,\ldots,n\}$ and $I_1=\{0,1,\ldots,n-1\}$  for $n\ge 1$. If $n$ is not in the range of $\lambda$, then $n$ is disconnected from any $i<n$, so $n$ must be in the range of $\lambda$. Therefore, without loss of generality, we cam assume that the range of $\lambda$ is the set $\{1,\ldots,n\}$. After renumbering, we can assume that $\lambda(i)=i+1$ for $i=0,\ldots,n-1$. Whence, $I_m=\{0,1,\ldots,n-m\}$ for $m=1,\ldots,n$ and $I_m=\emptyset$ for $m> n$.
\end{proof}

If the set $I_0$ is infinite and the kite $K^\lambda_{I_0,I_1}(\mathbf G)$ is subdirectly irreducible, then according to Proposition \ref{pr:5.5}, $|I_0|=\aleph_0$. In addition, $|I_1|=\aleph_0$.

\begin{proposition}\label{pr:5.9}
Let $\mathbf G$ be a non-trivial integral residuated lattice, $K^\lambda_{I_0,I_1}(\mathbf G)$ a subdirectly irreducible kite, and $I_0=\aleph_0$. Then one of the following situations happens:

\begin{itemize}
\item[{\rm(1)}] $I_0=I_1$ and $\lambda$ is bijective.
\item[{\rm(2)}] $\lambda$ is bijective and $|\lambda(I_1)\setminus I_1|=1$.
\item[{\rm(3)}] $I_0=I_1$ and $|I_1\setminus \lambda(I_1)|= 1$.
\end{itemize}
\end{proposition}

\begin{proof}
By Proposition \ref{pr:5.5}, we have $I_0=I_1\cup \lambda(I_1)$. If there are two distinct indices $i,j \in I_1\setminus \lambda(I_1)$, then $i$ and $j$  are disconnected. Therefore, $|I_1\setminus \lambda(I_1)|\le 1$. In a similar way, we have $|\lambda(I_1)\setminus I_1|\le 1$. We have the following four cases.

\vspace{2mm}
\noindent
{\it Case 1}. Assume $|I_1\setminus \lambda(I_1)|= 1$ and $|\lambda(I_1)\setminus I_1|=1$.

Choose $i\in I_1\setminus \lambda(I_1)$ and $j \in  \lambda(I_1)\setminus I_1$. By Theorem \ref{th:5.4}, $i$ and $j$ are connected. Therefore, there is an integer $m\ge 0$ such that $j=\lambda^m(i)$ (the second possibility $i=\lambda^m(j)$ is excluded because $j\notin I_1$). Define $k_n:=\lambda^n(i)$ for each $n=0,\ldots,m$. Take $k\in I_0 \setminus\{k_0,\ldots,k_m\}$ and consider $k_0=i$; then $k \in I_1 \cap \lambda(I_1)$. Since $k$ and $j$ are connected, there is an integer $s\ge 0$ such $\lambda^s(k)=j$. Then $\lambda^s(k)=\lambda^m(i)$. We have three subcases: (a) $s=m$ which yields $k=i$, a contradiction. (b) $s<m$, then $k=\lambda^{m-s}(i)$ which contradicts the choice of $k$. (c) $s>m$, then $\lambda^{s-m}(k)=i$. Since $s-m\ge 1$, we have $i\in \lambda(I_1)$ which is absurd, and Case 1 is excluded.

\vspace{2mm}
\noindent
{\it Case 2}. Assume $|I_1\setminus \lambda(I_1)|= 1$ and $|\lambda(I_1)\setminus I_1|=0$.

Then $\lambda(I_1)\subseteq I_1$ and $I_0=I_1\cup \lambda(I_1)=I_1$ which establishes (3).

\vspace{2mm}
\noindent
{\it Case 3}. Assume $|I_1\setminus \lambda(I_1)|= 0$ and $|\lambda(I_1)\setminus I_1|=1$.

Then $I_1\subseteq \lambda(I_1)$ which gives $I_0=I_1 \cup \lambda(I_1)=\lambda(I_1)$ and $\lambda$ is bijective which yields (2).

\vspace{2mm}
\noindent
{\it Case 4}. Assume $|I_1\setminus \lambda(I_1)|= 0= |\lambda(I_1)\setminus I_1|$.

Then $I_1\subseteq \lambda(I_1)\subseteq I_1$. Then $I_0=\lambda(I_1)=I_1$ and $\lambda$ is bijective which proves (1).
\end{proof}

\begin{remark}\label{re:5.10}
{\rm Let the conditions of Proposition \ref{pr:5.9} hold. If some of cases (1)--(3), holds, changing $I_1$ by $I_n$ and $I_0$ by $I_{n-1}$, then the same case holds also for $I_n$ and $I_{n-1}$ for each $n \ge 1$.
}
\end{remark}

Situations following from Propositions \ref{pr:5.9} are characterized as follows:

Case (1): $K^\lambda_{\mathbb Z, \mathbb Z}(\mathbf G)$ with $\lambda(i)=i+1$.

Case (2): $K^\lambda_{\mathbb N, \mathbb N\setminus\{0\}}(\mathbf G)$ with $\lambda(i)=i-1$.

Case (3):  $K^\lambda_{\mathbb N, \mathbb N}(\mathbf G)$ with $\lambda(i)=i+1$.

From the proof of Proposition \ref{pr:5.9}, one follows that kites from Case (1)--Case (3) are not mutually isomorphic.

Now we describe all subdirectly irreducible kites with $\mathbf G\ne \{e\}$ when $I_0$ is countably infinite.

\begin{theorem}\label{th:5.11}
Let $K^\lambda_{I_0,I_1}(\mathbf G)$ be a subdirectly irreducible kite, $\mathbf G$  non-trivial, and $|I_0|=\aleph_0$. Then $K^\lambda_{I_0,I_1}(\mathbf G)$ is isomorphic to just one of the following kites:
\begin{itemize}
\item[{\rm (1)}] $K^\lambda_{\mathbb Z, \mathbb Z}(\mathbf G)$ with $\lambda(i)=i+1$.

\item[{\rm (2)}] $K^\lambda_{\mathbb N, \mathbb N\setminus\{0\}}(\mathbf G)$ with $\lambda(i)=i-1$.

\item[{\rm (3)}]  $K^\lambda_{\mathbb N, \mathbb N}(\mathbf G)$ with $\lambda(i)=i+1$.
\end{itemize}
\end{theorem}

\begin{proof}
We use Proposition \ref{pr:5.9}.
Case (1). Let $\lambda$ be bijective. Then $\lambda$ is cyclic, otherwise there are $i,j \in I_0$ such that $j$ is not in the orbit $P(i) = \{\lambda^m(i)\colon \text{ such that }$ $ \lambda^m(i) \text { is defined},  m \in \mathbb Z\}$ of $i$, and $i$ and $j$ are disconnected, a contradiction. Hence, we can assume that $I_0=I_1 =\mathbb Z$ and $\lambda(i)=i+1$, and $K^\lambda_{I_0,I_1}(\mathbf G)$ is isomorphic to $K^\lambda_{\mathbb Z,\mathbb Z}(\mathbf G)$.

Case (2). If there is a unique $j\in I_0$ which does not belong to $I_1$, we can assume after renumbering that $I_1=\mathbb N\setminus\{0\}$, $I_0=\mathbb N$ and $\lambda(i)=i-1$.

Case (3). If there is a unique $j\in I_1$ which does not belong to the range of $\lambda$, then we can assume after renumbering that $I_0=\mathbb Z= I_1$ and $\lambda(i)=i+1$.
\end{proof}

In what follows we show that the following version of the Birkhoff Subdirect Representation theorem holds which says that every kite is subdirectly embeddable into a product of subdirectly irreducible kites.

\begin{proposition}\label{pr:5.12}
Let $\mathbf G$  be an integral residuated lattice which is subdirectly representable as $\mathbf G \le \prod_{s\in S}\mathbf G_s$, where each $\mathbf G_s$ is an integral residuated lattice. Then the kite $K^\lambda_{I_0,I_1}(\mathbf G)$ is subdirectly representable as $K^\lambda_{I_0,I_1}(\mathbf G) \le \prod_{s\in S} K^\lambda_{I_0,I_1}(\mathbf G_s)$.
\end{proposition}

\begin{proof}
The proof is straightforward and it is based on Proposition \ref{pr:5.2}.
\end{proof}

Before stating the next result, we recall that in the same way as there was defined connectedness of two points of the set $I_0$ and the connected component of $I_0$, we can define connectedness of any two points of the set $I_1$ and the connected component of $I_1$. If $C_0$ is a connected component of $I_0$, then the set $C_1:=\lambda^{-1}(C_0)$ is a connected component of $I_1$. Let $\mathcal I(I_0)$ and $\mathcal I(I_1)$ be the set of connected components of $I_1$ and $I_0$, respectively. Then $\mathcal I(I_1)= \{\lambda^{-1}(C)\colon C \in \mathcal I(I_0)\}$.

\begin{theorem}\label{th:5.13}
Every kite is a subdirect product of a system of subdirectly irreducible kites.
\end{theorem}

\begin{proof}
Let $K^\lambda_{I_0,I_1}(\mathbf G)$ be an arbitrary kite associated with an integral residuated lattice $\mathbf G$. If $\mathbf G$ is trivial, by Example \ref{sub:5}, the kite is isomorphic to the kite $\mathbf Z^-$ which is subdirectly irreducible and the statement is trivially satisfied.

Now let $\mathbf G$ be non-trivial and let $\mathcal I(I_0)$, $\mathcal I(I_1)$ be the set of connected components of $I_0$ and $I_1$, respectively.  For each $C_0 \in \mathcal I(I_1)$, let $C_1=\lambda^{-1}(C_0)$, and let $\lambda_{C_0}:C_1\to C_0$ be the restriction of $\lambda$ onto $C_1$, $C_1 \in \mathcal I(I_1)$. Given $C_0\in \mathcal I(I_0)$, we define the new kite $K^{\lambda_{C_0}}_{C_0,C_1}(\mathbf G)$. In addition, we define the set
$N_{C_0}$ as the set of all elements $\langle x_i\colon i \in I_0\rangle\in G^{I_0}$ such that $i\in C_0$ implies $x_i=e$. Then $N_{C_0}$ is a normal filter of $K^\lambda_{I_0,I_1}(\mathbf G)$, and it is possible to show that $K^\lambda_{I_0,I_1}(\mathbf G)/N_{C_0}$ is isomorphic to $K^{\lambda_{C_0}}_{C_0,C_1}(\mathbf G)$.

As every two distinct connected components of $I_0$ are mutually disjoint, we have $\bigcap\{N_{C_0}\colon C_0\in \mathcal I(I_0)\}=\{1\}$ which proves that $K^\lambda_{I_0,I_1}(\mathbf G)$ is subdirectly embeddable into the product of the system of kites  $\{K^{\lambda_{C_0}}_{C_0,C_1}(\mathbf G)\colon C_0 \in \mathcal I(I_0)\}$.

To finish the proof, we have to show that every $K^{\lambda_{C_0}}_{C_0,C_1}(\mathbf G)$ is a subdirect product of subdirectly irreducible kites.  For $\mathbf G$ there is a system of integral subdirectly irreducible residuated lattices $\{\mathbf G_s\colon s \in S\}$ such that $\mathbf G \le \prod_{s \in S} \mathbf G_s$, which by Proposition \ref{pr:5.12} proves that every $K^{\lambda_{C_0}}_{C_0,C_1}(\mathbf G) \le \prod_{s\in S} K^{\lambda_{C_0}}_{C_0,C_1}(\mathbf G_s)$.  Using the criterion Theorem \ref{th:5.4}, every $K^{\lambda_{C_0}}_{C_0,C_1}(\mathbf G_s)$ for each $s \in S$ is subdirectly irreducible, which establishes the statement.
\end{proof}

The latter theorem implies directly the following result:

\begin{theorem}\label{th:5.14}
The variety $\mathsf K$ of integral residuated lattices generated by all kites is generated by all subdirectly irreducible kites.
\end{theorem}

\section{Infinite-dimensional and Finite-dimensional Kites}%6

In this section we show that the class of all finite-dimensional kites generates the variety $\mathsf K$ of integral residuated lattices generated by all kites.

A finite-dimensional kite $K^\lambda_{I_0,I_1}(\mathbf G)$ is said to be $n$-{\it dimensional}, if $|I_0|=n$ for some integer $n\ge 0$. We write $\mathcal K_n$ the class of $n$-dimensional kites, and let $\mathsf K_n$ be the variety of integral residuated lattices generated by $\mathcal K_n$.

Our method will be based on embedding every kite from Theorem \ref{th:5.11} into some product of finite-dimensional kites. If $\mathbf G=\{e\}$, then by Example \ref{sub:5}, $K^\lambda_{I_0,I_1}(\mathbf G)\cong \mathbf Z^-$, so the kite $K^\lambda_{I_0,I_1}(\mathbf G)$ belongs to the variety generated by $\mathsf K_1$. Due to Theorems \ref{th:5.11} and \ref{th:5.13}, it is enough to assume that that $\mathbf G$ is non-trivial and $I_0$ is countably infinite.

Let $\mathbf G$ be a non-trivial integral residuated monoid. First we start with embedding the kite $K^\lambda_{\mathbb N,\mathbb N}(\mathbf G)$ with $\lambda(i)=i+1$ into the direct product $\prod_{k=1}^\infty K^{\lambda_k}_{k+1,k}(\mathbf G)$, where $\lambda_k(i)=i+1$. Then on one side, every element $x\in K^\lambda_{\mathbb N,\mathbb N}(\mathbf G)$ is expressible in the form $x = \langle x_i \colon i \in I_n\rangle$, where $I_n=\mathbb N$ for each $n\ge 0$.

On the other hand, each kite of the form $K^{\lambda_k}_{k+1,k}(\mathbf G)$
for $k\ge 1$  can be characterized by the sequence of subsets $\{I^k_n\colon n \ge 0\}$, where $I^k_n=\{0,\ldots, k-n\}$ for $n=0,\ldots,k$ and $I^k_n=\emptyset $ for $n>k$, and with an injective mapping $\lambda_k:I^k_1=\{0,\ldots,k-1\}\to I^k_0=\{0,\ldots,k\}$ defined $\lambda_k(i)=i+1$, $i=0,\ldots,k-1$.

Hence, we characterize an element $x \in \prod_{k=1}^\infty K^{\lambda_k}_{k+1,k} (\mathbf G)$  by a sequence $\langle\langle x^k_i\colon i \in I^k_{m_k}\rangle\colon k\ge 1\rangle $, where $x^k_i \in G$. Define a mapping $\phi_1: K^\lambda_{\mathbb N,\mathbb N}(\mathbf G) \to \prod_{k=1}^\infty K^{\lambda_k}_{k+1,k} (\mathbf G)$ as follows
$$
\phi_1(\langle x_i\colon i \in I_m\rangle):=\langle \langle x_i\colon i \in I^k_m\rangle \colon k \ge 1\rangle, \eqno(6.1)
$$
where if $I^k_m =\emptyset$, we put as before, $\langle x_i\colon i \in \emptyset \rangle :=\langle e\colon i \in \emptyset \rangle$.
Then

\begin{eqnarray*}
& &\phi_1(\langle x_i\colon i \in I_m\rangle \cdot \langle y_i\colon i \in I_n\rangle ) = \phi_1( \langle x_{i+n}y_i\colon i \in I_{m+n}\rangle)\\
& & = \langle \langle x_{i+n}y_i\colon i \in I^k_{n+m}\rangle\colon k\ge 1\rangle,\\
& &\phi_1(\langle x_i\colon i \in I_m\rangle) \cdot \phi_1(\langle y_i\colon i \in I_n\rangle ) = \langle \langle x_i \colon i \in I^k_m\rangle \colon k \ge 1\rangle \cdot \langle \langle y_i \colon i \in I^k_n\rangle \colon k \ge 1\rangle\\
& & = \langle \langle x_{i+n}y_i \colon i \in I^k_{m+n}\rangle \colon k \ge 1\rangle,
\end{eqnarray*}
so that $\phi_1$ preserves product, and $\phi_1(1)=\phi(\langle e: i \in I_0\rangle)=\langle \langle e:i\in I^k_0\rangle \colon k\ge 1\rangle $.

%$\sslash$, and $\bss$,

For $m\le n$, we have  $\langle y_i\colon i \in I_n\rangle \sslash \langle x_i\colon i \in I_m\rangle =\langle (y \sslash x)_i \colon i \in I_{n-m}\rangle$, where

\begin{align*}
(y\sslash x)_i&= \begin{cases}e & \text{ if } 0\le i <m\\
y_{i-m}\rd x_{i-m} & \text{ if } m\le i,
\end{cases}
\text{ for } i \in I_{n-m}=\mathbb N,
\end{align*}
i.e. $\langle (y \sslash x)_i \colon i \in I_{n-m}\rangle =\langle e,\ldots,e,y_{-m}\rd x_{-m}, y_{1-m}\rd x_{1-m},\ldots, y_{i-m}\rd x_{i-m},\ldots \rangle$, where $e$'s we have $m$-times if $m>0$, otherwise, there is no $e$.

For the product we have $\langle x^k_i\colon i \in I^k_m \rangle \cdot \langle y^k_i\colon i \in I^k_n\rangle =\langle x^k_ny^k_0, \ldots, x^k_{k-m}y^k_{k-m-n}\rangle$ if $ k\ge m+n$.

In addition, for each integer $k\ge 1$, we have $\langle y_i\colon i \in I^k_n\rangle \sslash \langle x_i\colon i \in I^k_m\rangle =\langle (y \sslash x)^k_i \colon i \in I^k_{n-m}\rangle$, where

\begin{align*}
(y\sslash x)^k_i&= \begin{cases}
e & \text{ if } 0\le i <m\\
y_{i-m}\rd x_{i-m} & \text{ if } m\le i \le k-n+m,
\end{cases}
%\text{ for } i \in I^k_{n-m}=\{0,\ldots,k-n+m\}
\end{align*}
which entails $\phi(\langle y_i\colon i \in I_n\rangle \sslash \langle x_i\colon i \in I_m\rangle) =\phi_1(\langle y_i\colon i \in I_n\rangle) \sslash \phi_1(\langle x_i\colon i \in I_m\rangle)$.

Similarly, for $n\le m$, we have $\langle y_i\colon i \in I_n\rangle \bss \langle x_i\colon i \in I_m\rangle =\langle (y \bss x)_i \colon i \in I_{m-n}\rangle=\langle y_{m-n}\ld x_0, y_{1+m-n}\ld x_1,\ldots, y_{i+m-n}\ld x_i,\ldots\rangle$, and for each integer $k\ge 1$, we have $I^k_{m-n}=\{0,\ldots, k-m+n\}$ if $k\le m-n$, $I^k_{m-n}=\emptyset$ if $m-n<k$, and
$\langle y_i\colon i \in I^k_n\rangle \bss \langle x_i\colon i \in I^k_m\rangle =\langle (y \bss x)^k_i \colon i \in I^k_{m-n}\rangle$, where
\begin{align*}
&  (y \bss x)^k_i  =\begin{cases} y_{i+m-n}\ld x_i& \text{if } i \in I^k_m=\{0,\ldots,k-m\}\\
 e& \text{if } i\in I^k_{m-n}\setminus I^k_m=\{k-m+1,\ldots,k-m+n\}
\end{cases}
\end{align*}
if $m \le k$ and $(y \bss x)_i=e$ if $k<m$ and $i\in I^k_{m-n}=\emptyset$. Hence, for each $k>m\ge n$,  $\langle (y \bss x)^k_i\colon i \in I^k_m\rangle =\langle  y_{m-n}\ld x_0,\ldots, y_{i+m-n}\ld x_i, \ldots, y_{k-n}\ld x_{k-m},e,\ldots,e\rangle$, where $e$ is at the end of the sequence $n$-times. Consequently $\phi_1$ does not preserves $\bss$ and $\phi_1$ is no embedding. In what follows, we introduce a congruence $\approx$ such that $\phi_1/\approx$ will be an embedding.

Let $K_1$ be the subset of $\prod_{k=1}^\infty K^{\lambda_k}_{k+1,k} (\mathbf G)$ consisting of elements of the form $\langle\langle x^k_i\colon i\in I^k_m\rangle \colon k\ge 1\rangle$, where $m\ge 0$. From the above calculation, we see that $K_1$ is an integral residuated subalgebra of the direct product $\prod_{k=1}^\infty K^{\lambda_k}_{k+1,k} (\mathbf G)$.

We define a relation $\approx$ between  elements of $K_1$ as follows
$x=\langle\langle x^k_i\colon i\in I^k_m\rangle \colon k\ge 1\rangle \approx y=\langle\langle y^k_i\colon i\in I^k_n\rangle \colon k\ge 1\rangle$ iff $m=n$ and there exist integers $k_0\ge m$ and $d$ with $0\le d\le k_0-m$ such that for each $k\ge k_0$,
$x^k_i=y^k_i$ for $i=0,\ldots, k-m-d$.

\begin{proposition}\label{pr:6.1}
The relation $\approx$ is a congruence on $K_1$.
\end{proposition}

\begin{proof}
Reflexivity and symmetry are obvious. To prove transitivity, suppose $x= \langle\langle x^k_i\colon i\in I^k_m\rangle \colon k\ge 1\rangle\approx  y= \langle\langle y^k_i\colon i\in I^k_m\rangle \colon k\ge 1\rangle\approx z=\langle\langle z^k_i\colon i\in I^k_m\rangle \colon k\ge 1\rangle$.
By definition, there are $k_1, k_2\ge m$, $0\le d_1\le k_1-m$, $d_2\le k_2-m$  such that for each $k\ge k_0=\max\{k_1,k_2\}$, $x^k_i=y^k_i$ for $i=0,\ldots,k-m-d_1$ and $y_i=z_i$ for $i=0,\ldots, k-m-d_2$. If we put $d=\max\{d_1,d_2\}$, we have $x^k_i=z^k_i$ for $i=0,\ldots,k-m-d$ and transitivity of $\approx$ is established.

Now let $x=\langle\langle x^k_i\colon i\in I^k_m\rangle \colon k\ge 1\rangle \approx y = \langle\langle y^k_i\colon i\in I^k_m\rangle \colon k\ge 1\rangle$ and $u = \langle\langle u^k_i\colon i\in I^k_n\rangle \colon k\ge 1\rangle \approx v= \langle\langle v^k_i\colon i\in I^k_n\rangle \colon k\ge 1\rangle$. There are $k_1 \ge m$, $k_2\ge n$, $d_1,d_2$ with $0\le d_1 \le k_1-m$, $0\le d_2 \le k_2-n$ such that for each $k\ge k_1$, $x^k_i=y^k_i$ for $i=0,\ldots, k-m-d_1$ and for each $k\ge k_2$,  $u^k_i=v^k_i$ for $i=0,\ldots, k-n-d_2$.

$x\cdot u\approx y\cdot v$:
Let $d=\max\{d_1,d_2\}$ and let $k\ge k_0:= \max\{k_1,k_2\}+m+n+d$. Then $k-m-n\ge \max\{k_1,k_2\}+m+n+d -m-n = \max\{k_1,k_2\}+d >n$, and $k-m-n-d\le k-n-d_1$, $k-m-n-d\le k-m-d_2$, so that $x^k_i=y^k_i$ for $i=n,\ldots, k-m-n-d$, and $u^k_i=v^k_i$ for $i=0,\ldots, k-m-n-d$. So that $x^k_{n+i}u^k_i= y^k_{n+i}v^k_i$ for $i=0,\ldots,k-m-n-d$, i.e. $x\cdot u\approx y\cdot v$.

To establish that $\approx$ preserves divisions, assume first $\sslash$ and $m\le n$. Then for $u\sslash x =\langle (u^k\sslash x^k)_i\colon I^k_{n-m}\rangle$ and $v\sslash y =\langle (v^k\sslash y^k)_i\colon I^k_{n-m}\rangle$, where

\begin{align*}
(u^k\sslash x^k)_i&= \begin{cases}
e & \text{ if } 0\le i <m\\
u^k_{i-m}\rd x^k_{i-m} & \text{ if } m\le i \le k-n+m,
\end{cases}
\end{align*}
and
\begin{align*}
(v^k\sslash y^k)_i&= \begin{cases}
e & \text{ if } 0\le i <m\\
v^k_{i-m}\rd y^k_{i-m} & \text{ if } m\le i \le k-n+m.
\end{cases}
\end{align*}

If we take $k\ge k_0:=\max\{k_1,k_2\} +m+n-d$, then $k-n+m-d\ge \max\{k_1,k_2\} +m +n-d+m- n- d= \max\{k_1,k_2\}+2m\ge m$, so that $x^k_{i-m}=y^k_{i-m}$  and $u^k_{i-m}=v^k_{i-m}$ for $i=m,\ldots,k-n+m-d$ which yields $u^k_{i-m}\rd x^k_{i-m}= v^k_{i-m}\rd y^k_{i-m}$ for $i=m,\ldots, k-n+m-d$. Consequently, $(u^k\sslash x^k)_i = (v^k\sslash y^k)_i$ for $i=0,\ldots, k-n+m-d$, and $u\sslash x \approx v\sslash y$.

Now we establish that $\approx$ preserves $\bss$. So let $n\le m$. Then $u\bss x=\langle (u^k\bss x^k)_i\colon i \in I^k_{m-n}\rangle$ and $v\bss y=\langle (v^k\bss y^k)_i\colon i \in I^k_{m-n}\rangle$, where

\begin{align*}
&  (u^k \bss x^k)_i  =\begin{cases} u^k_{i+m-n}\ld x^k_i& \text{if } i \in \{0,\ldots,k-m\}\\
 e& \text{if } i\in \{k-m+1,\ldots,k-m+n\}
\end{cases}
\end{align*}
if $m \le k$ and $(u \bss x)_i=e$ if $k<m$ and $i\in I^k_{m-n}=\emptyset$. Hence, for each $k>m\ge n$,  $\langle (u \bss x)^k_i\colon i \in I^k_{m-n}\rangle =\langle  u^k_{m-n}\ld x^k_0,\ldots, u^k_{i+m-n}\ld x^k_i, \ldots, u^k_{k-n}\ld x^k_{k-m},e,\ldots,e\rangle$, where $e$ is at the end of the sequence $n$-times. Similarly,  for each $k>m\ge n$,  $\langle (v \bss y)_i\colon i \in I^k_{m-n}\rangle= \langle  v^k_{m-n}\ld y^k_0,\ldots, v^k_{i+m-n}\ld y^k_i, \ldots, v^k_{k-n}\ld y^k_{k-m},e,\ldots,e\rangle$.

Set $d=\max\{d_1,d_2,n\}$ and let $k_0$ be an integer such that $k_0\ge \max\{k_1,k_2, 2(m-n) + d\}$. Then for $k \ge k_0$, we have $k-m+n-d\ge 2(m-n)+d -m+n-d=m-n$, so that if $i=0,\ldots,k-m+n-d$, then $u^k_{i+m-n}=v^k_{i+m-n}$ as well for $i=0,\ldots, k-m-d$. Hence, $x^k_i=y^k_i$ for $i=0,\ldots,k-m-d$ which yields, $u^k_{i+m-n}\ld x^k_i= v^k_{i+m-n}\ld y^k_i$ for $i=0,\ldots, k-m-d$. Finally, $(u^k \bss x^k)_i = (v^k \bss y^k)_i$ for $i=0,\ldots, k-m+n-(d+n)$, i.e. $u \bss x \approx v \bss y$.

Summarizing all the above cases, we see that $\approx$ is a congruence of $K_1$.
\end{proof}

\begin{proposition}\label{pr:6.2}
Let $\Phi_1: K^\lambda_{I_0,I_1}(\mathbf G) \mapsto \phi_1(K^\lambda_{I_0,I_1}(\mathbf G))/\approx$. Then $\Phi_1$ is an embedding of $K^\lambda_{I_0,I_1}(\mathbf G)$ into $K_1/\approx$.
\end{proposition}

\begin{proof}
As we have seen above, $\phi_1$ preserves product and $\sslash$. Now let $n\le m$ and let $x =\langle x_i\colon i \in I_m\rangle$ and $y= \langle y_i\colon i \in I_n\rangle$. Then  $\langle y_i\colon i \in I_n\rangle \bss \langle x_i\colon i \in I_m\rangle =\langle (y \bss x)_i \colon i \in I_{m-n}\rangle=\langle y_{m-n}\ld x_0, y_{1+m-n}\ld x_1,\ldots, y_{i+m-n}\ld x_i,\ldots\rangle$, and $\phi_1(x)=\langle \langle x^k_i\colon i \in I^k_m\rangle \colon k\ge 1\rangle$, $\phi_1(y)=\langle \langle y^k_i\colon i \in I^k_n\rangle \colon k\ge 1\rangle$, $\phi_1(y\bss x)= \langle \langle z^k_i\colon i \in I^k_{m-n}\rangle\colon k\ge 1\rangle$, where $x^k_i= y^k_i =z^k_i= e$ if $k<m-n$ and $x^k_i=x_i$, $y^k_i=y_i$, $z^k_i=y_{i+m-n}\ld x_i$ if $m-n\le k$.

On the other hand, for $\langle (y\bss x)_i\colon i \in I^k_{m-n}\rangle$,
we have
\begin{align*}
&  (y \bss x)_i  =\begin{cases} y_{i+m-n}\ld x_i& \text{if } i \in I^k_m=\{0,\ldots,k-m\}\\
 e& \text{if } i\in I^k_{m-n}\setminus I^k_m=\{k-m+1,\ldots,k-m+n\}
\end{cases}
\end{align*}
if $m \le k$ and $(y \bss x)_i=e$ if $k<m$ and $i\in I^k_{m-n}=\emptyset$. Hence, for each $k>m\ge n$,  $\langle (y\bss  x)_i\colon i \in I^k_{m-m}\rangle =\langle  y_{m-n}\ld x_0,\ldots, y_{i+m-n}\ld x_i, \ldots, y_{k-n}\ld x_{k-m},e,\ldots,e\rangle$, where $e$ is at the end of the sequence $n$-times.

If $m-n\le k$, then $\langle z^k_i\colon i\in I^k_{m-n}\rangle = \langle  y_{m-n}\ld x_0,\ldots, y_{i+m-n}\ld x_i, \ldots, y_{k}\ld x_{k-m+n}\rangle$.
Comparing the latter two vectors, we see that if $d=n$, and $k\ge k_0=m+1$, then $z^k_i= y_{i+m-n}$ for $i=0,\ldots, k-m+n-d$, i.e. $\phi_1(y\bss x)\approx \phi_1(y)\bss \phi_1(x)$ and $\Phi_1(y\bss x)=\Phi_1(y)\bss \Phi_1(x)$.

We have established that $\Phi_1$ is a homomorphism. We claim that $\Phi_1$ is injective. Let $\Phi_1(x)=\Phi_1(y)$, where $x=\langle x_i\colon i\in I_m\rangle$ and $y= \langle y_i\colon i\in I_n\rangle$.
Then $m=n$,  and $\phi_1(x)=\langle x^k_i\colon i \in I^k_m\rangle$ and $\phi_1(y) = \langle y^k_i\colon i \in I^k_m\rangle$ where $x^k_i= y^k_i= e$ if $k<m-n$ and $x^k_i=x_i$, $y^k_i=y_i$  if $m-n\le k$.
If $k\ge k_0=m$ and $d=0$, then $x_i=x^k_i=y^k_i=y_i$. Hence, $x_i=y_i$ for each $i\ge 0$ and $x=y$.
\end{proof}

Now we take a kite  of the form $K^\lambda_{\mathbb Z, \mathbb Z}(\mathbf G)$ with $\lambda(i)=i+1$ with $I_0=\mathbb Z=I_1$, where $\mathbf G$ is a non-trivial integral residuated lattice. Then $I_m=\mathbb Z$ for every $m\ge 0$. For each integer $k\ge 0$, let $I^k_0$ be the $2k+1$-element set $\mathbb Z/(2k+1)\mathbb Z$ which is the additive group. We represent the set $I^k_0$ as $I^k_0=\{-k,-k+1,\ldots,-1,0,1, \ldots, k-1,k\}$ and we set $I^k_1=I^k_0$ with $\lambda_k(i)=i+1 (\mbox{mod } 2k+1)$. The labeling of elements from this $2k+1$-element sets are counted as the additive group $\mathbb Z/(2k+1)\mathbb Z$.  Then $I^k_m=I^k_0$ for each $m\ge 1$. We set $K^{\lambda_k}_k(\mathbf G):= K^{\lambda_k}_{I^k_0,I^k_1}(\mathbf G)$ for each $k\ge 1$. Define a mapping
$\phi_2: K^\lambda_{\mathbb Z, \mathbb Z}(\mathbf G) \to \prod_{k=0}^\infty K^{\lambda_k}_k(\mathbf G)$ by
$$\phi_2(\langle x_i\colon i \in I_m\rangle)=\langle \langle x_{\lambda_k(i)}\colon i \in I^k_m\rangle\colon k\ge 0\rangle.
\eqno(6.2)
$$

Let $K_2$ be the subset of $\prod_{k=0}^\infty K^{\lambda_k}_{k} (\mathbf G)$ consisting of elements of the form $\langle\langle x_i\colon i\in I^k_m\rangle \colon k\ge 1\rangle$, where $m\ge 0$. Then $K_2$ is a subalgebra of the product $\prod_{k=0}^\infty K^{\lambda_k}_{k} (\mathbf G)$.

\begin{proposition}\label{pr:6.3}
The mapping $\phi_2$ is an embedding of $K^\lambda_{\mathbb Z, \mathbb Z}(\mathbf G)$ into $K_2$.
\end{proposition}

\begin{proof}
Let $x = \langle x_i \colon i \in I_m\rangle$ and $y = \langle y_i \colon i \in I_n\rangle$. Then $\phi_2(x)= \langle \langle x_{\lambda_k(i)}\colon i \in I^k_m\rangle \colon k\ge 0\rangle$ and $\phi_2(y)= \langle \langle y_{\lambda_k(i)}\colon i \in I^k_n\rangle \colon k\ge 0\rangle$.

Take product $x\cdot y = \langle x_{\lambda_k^n(i)}y_i\colon i \in I_{m+n}\rangle$. Then $\phi_2(x)\cdot \phi_2(y) = \langle \langle x_{\lambda^n_k(i)}y_i\colon i \in I^k_{m+n}\rangle \colon k\ge 0\rangle = \phi_2(x\cdot y)$.

If  $n\le m$, then $y \bss x=\langle y_{i+m-n}\colon i \in I_{m-n}$, and
$\phi_2(y)\bss \phi_2(x)=\langle \langle y_{\lambda_k(i)}\colon i \in I^k_n\rangle \colon k\ge 0\rangle \bss \langle \langle x_{\lambda_k(i)}\colon i \in I^k_m\rangle \colon k\ge 0\rangle =\langle \langle y_{\lambda_k^{m-n}(i)}\ld x_i \colon i \in I^k_{m-n}\rangle \colon k\ge 0\rangle= \phi_2(y \bss x)$.

If $n\ge m$, then $y\sslash x =\langle y_{i-m}\rd x_{i-m}\colon i \in I_{n-m}\rangle \colon k\ge 0\rangle$, and $\phi_2(y) \sslash \phi_2(x)=
\langle \langle y_{\lambda_k(i)}\colon i \in I^k_n\rangle \colon k\ge 0\rangle \sslash \langle \langle x_{\lambda_k(i)}\colon i \in I^k_m\rangle \colon k\ge 0\rangle =\langle \langle y_{\lambda_k^{-m}(i)}\rd x_{\lambda_k^{-m}(i)} \colon i \in I^k_{n-m}\rangle \colon k\ge 0\rangle= \phi_2(y \sslash x)$.

We have proved that $\phi_2$ is a homomorphism. It is straightforward to see that $\phi_2$ is injective, which proves the proposition.
\end{proof}

It remains the last case, the infinite kite $K^\lambda_{\mathbb N, \mathbb N\setminus\{0\}}(\mathbf G)$ with $\lambda(i)=i-1$. Then $I_0=\{0,1,\ldots, \}$,  $I_1=\{1,2,\ldots\}$ and $I_n=\{n,n+1,\ldots\}$. For any integer $k\ge 1$, we set $I^k_0=\{0,\ldots,k\}$, $I^k_1=\{1,\ldots,k\}$ with $\lambda_k(i)=i-1$. Then $I^k_m=\{m,\ldots,k\}$ if $m\le k$, otherwise, $I^k_m=\emptyset$. Let $K_{k,k-1}^{\lambda_k}(\mathbf G):= K_{I^k_0,I^k_1}^{\lambda_k}(\mathbf G)$ for each $k\ge 1$, and define the direct product $\prod_{k=1}^\infty K_{k,k-1}^{\lambda_k}(\mathbf G)$.

Choose $x = \langle x_i \colon i \in I_m\rangle$ and $y = \langle y_i \colon i \in I_n\rangle$. Then $x\cdot y = \langle x_{i-n}y_i\colon i \in I_{m+n}\rangle = \langle x_my_{m+n},x_{m+1}y_{m+n+1},\ldots, x_{i-n}y_i,\ldots \rangle$. If $n\le m$, then $y\bss x=\langle (y\bss x)_i\colon i \in I^{m-n}\rangle$, where

\begin{align*}
&  (y \bss x)_i  =\begin{cases}
 e& \text{if } i\in I_{m-n}\setminus I_m=\{m-n,\ldots,m-1\}\\
y_{i-m+n}\ld x_i& \text{if } i \in I_m=\{m,m+1,\ldots\},
\end{cases}
\end{align*}
i.e., $y\bss x =\langle e,\ldots,e, y_{n}\ld x_m, \ldots, y_{i-m+n}\ld x_i,\ldots\rangle$, where at the beginning of the foregoing vector the element $e$ is $n$-times.

Similarly, if $m\le n$, then $y\sslash x = \langle (y\rd x)_i\colon i \in I_{n-m}\rangle$, where

$$ (y\sslash x)_i= y_{i+m}\rd x_{i+m} \text{ for } i\in \{n-m,n-m+1,\ldots\},
$$
i.e. $y\sslash x =\langle y_n\rd x_n, \ldots, y_{i+m}\rd x_{i+m}, \ldots \rangle$.

Now let $k\ge m+n$. Then $\langle x^k_i\colon i\in I^k_m\rangle \cdot
\langle y^k_i\colon i\in I^k_n\rangle = \langle x^k_{i-n}y^k_i\colon i \in I^k_{m+n}\rangle =\langle x^k_my^k_{m+n},\ldots, x^k_{i-n}y^k_i,\ldots, x^k_{k-n}y^k_k\rangle$.

Let $n\le m\le k$. Then  $\langle y^k_i\colon i\in I^k_n\rangle \bss
\langle x^k_i\colon i\in I^k_m\rangle =\langle (y^k\bss x^k)_i\colon i \in I^k_{m-n}\rangle$, where

\begin{align*}
&  (y^k \bss x^k)_i  =\begin{cases}
 e& \text{if } i\in I^k_{m-n}\setminus I^k_m=\{m-n,\ldots,m-1\}\\
y^k_{i-m+n}\ld x^k_i& \text{if } i \in I^k_m=\{m,m+1,\ldots,k\},
\end{cases}
\end{align*}
i.e., $\langle y^k_i\colon i\in I^k_n\rangle \bss
\langle x^k_i\colon i\in I^k_m\rangle =\langle e,\ldots,e, y^k_{n}\ld x^k_m, \ldots, y^k_{i-m+n}\ld x^k_i,\ldots, y^k_{k-m+n}\ld x^k_k\rangle$, where at the beginning of the foregoing vector the element $e$ is $n$-times.

Similarly, if $m\le n\le k$, then $\langle y^k_i\colon i\in I^k_n\rangle \sslash
\langle x^k_i\colon i\in I^k_m\rangle = \langle (y^k\rd x^k)_i\colon i \in I_{n-m}\rangle$, where

\begin{align*}
&  (y^k \sslash x^k)_i  =\begin{cases}
y^k_{i+m}\rd x^k_{i+m}& \text{if } i \in I^k_{n-m}\cap \lambda^m(I^k_n)=\{n-m,\ldots,k-m\}\\
e& \text{if } i\in I^k_{n-m}\setminus \lambda^m(I^k_n)=\{k-m+1,\ldots,k\},
\end{cases}
\end{align*}
i.e., $\langle y^k_i\colon i\in I^k_n\rangle \sslash
\langle x^k_i\colon i\in I^k_m\rangle  =\langle y^k_n\rd x^k_n,\ldots, y^k_{i+m}\rd x^k_{i+m},\ldots, y^k_{k}\rd x^k_k,e,\ldots,e\rangle$, where $e$ is $m$-times.

Let $K_3$ be the subset of $\prod_{k=1}^\infty K^{\lambda_k}_{k,k-1} (\mathbf G)$ consisting of elements of the form $\langle\langle x^k_i\colon i\in I^k_m\rangle \colon k\ge 1\rangle$, where $m\ge 0$. Then $K_3$ is a subalgebra of the product $\prod_{k=1}^\infty K^{\lambda_k}_{k,k-1} (\mathbf G)$. On $K_3$ we define a relation $\approx$ as follows:

Two vectors $x=\langle\langle x^k_i\colon i\in I^k_m\rangle \colon k\ge 1\rangle \approx y=\langle\langle y^k_i\colon i\in I^k_n\rangle \colon k\ge 1\rangle$ iff $m=n$ and there exist integers $k_0\ge m$ and $d$ with $0\le d\le k_0-m$ such that for each $k\ge k_0$, $x^k_i=y^k_i$ for $i=m,\ldots, k-d$.

\begin{proposition}\label{pr:6.4}
The relation $\approx$ is a congruence on the subalgebra $K_3$.
\end{proposition}

\begin{proof}
In the same way as in the proof of Proposition \ref{pr:6.1}, we can establish that $\approx$ is an equivalency.

Now let $x=\langle\langle x^k_i\colon i\in I^k_m\rangle \colon k\ge 1\rangle \approx y = \langle\langle y^k_i\colon i\in I^k_m\rangle \colon k\ge 1\rangle$ and $u = \langle\langle u^k_i\colon i\in I^k_n\rangle \colon k\ge 1\rangle \approx v= \langle\langle v^k_i\colon i\in I^k_n\rangle \colon k\ge 1\rangle$. There are $k_1 \ge m$, $k_2\ge n$, $0\le d_1\le k_1-m$ and $0\le d_2 \le k_2-n$ such that for each $k\ge k_1$, $x^k_i=y^k_i$ for $i=m,\ldots, k-d_1$ and for each $k\ge k_2$,  $u^k_i=v^k_i$ for $i=n,\ldots, k-d_2$.

$x\cdot u\approx y\cdot v$: Let $d=\max\{d_1,d_2\}$ and let $k\ge k_0:= \max\{k_1,k_2\}+m+n+d$. Then $k-d\ge \max\{k_1,k_2\}+m+n+d -d= \max\{k_1,k_2\}+m+n\ge m+n\ge m,n$. So that $x^k_i=y^k_i$ for $i=m,\ldots, k-d$, $u^k_i=v^k_i$ for $i=n,\ldots, m+n,\ldots,k-d$, which gives $x^k_iu^k_{i+n}$ for $i=m,\ldots, k-d$, i.e. $x\cdot u\approx y\cdot v$.

For the division $\bss$, let us assume $n\le m \le k$. Then $\langle u^k_i\colon i\in I^k_n\rangle \bss
\langle x^k_i\colon i\in I^k_m\rangle =\langle e,\ldots,e, u^k_{n}\ld x^k_m, \ldots, u^k_{i-m+n}\ld x^k_i,\ldots, u^k_{k-m+n}\ld x^k_k\rangle$, where at the beginning of the foregoing vector the element $e$ is $n$-times. If $k\ge k_0:=\max\{k_1,k_2\}+m+n+d$, then $u^k_i=v^k_i$ for $i=n,\ldots,m,\ldots, k-d$ and $x^k_i=v^k_i$ for $i=m,\ldots, k$. Hence, $u^k_{i-m+n}\ld x^k_i = v^k_{i-m+n}\ld y^k_i$ for $i=m,\ldots, k$. Then $x\bss u\approx y\bss v$.

For the division $\sslash$, let us assume $m\le n \le k$. Then we have $\langle u^k_i\colon i \in I^k_n \rangle \sslash \langle x^k_i \colon i \in I^k_m\rangle =\langle u^k_n\rd x^k_n,\ldots, u^k_{i+m}\rd x^k_{i+m},\ldots, u^k_{k}\rd x^k_k,e,\ldots,e\rangle$, where $e$ is $m$-times.
We have $x^k_i=y^k_i$ for $i=m,\ldots,n,\dots,k-d$ and $u^k_i=v^k_i$ for $i=n,\ldots,k-d$, which yields $u^k_{i}\rd x^k_{i+m}=v^k_{i}\rd y^k_{i+m}$ for $i=m,\ldots,k-d $ which easily entails that $u\sslash x\approx v\sslash y$.
\end{proof}

\begin{proposition}\label{pr:6.5}
Define a mapping $\Phi_3$ which maps $K^\lambda_{\mathbb N, \mathbb N\setminus\{0\}}(\mathbf G)$ into $K_3/\approx$  by
$$
\Phi_3(\langle x_i\colon i \in I_m\rangle):= \langle\langle x_i\colon i \in I^k_m\rangle\colon k \ge 1\rangle/\approx.
$$
Then $\Phi_3$ is an embedding of $K^\lambda_{\mathbb N, \mathbb N\setminus\{0\}}(\mathbf G)$ into $K_3/\approx$.
\end{proposition}

\begin{proof}
Choose $x = \langle x_i \colon i \in I_m\rangle$ and $y = \langle y_i \colon i \in I_n\rangle$. Then $\Phi_3(x) =\langle \langle x_i\colon i \in I_m\rangle \colon k\ge 1\rangle$, $\Phi_3(y) =\langle \langle y_i\colon i \in I_n\rangle \colon k\ge 1\rangle$.

For the product, we have $x\cdot y = \langle x_my_{m+n},x_{m+1}y_{m+n+1},\ldots, x_{i-n}y_i,\ldots \rangle$, and $\Phi_3(x\cdot y)= \langle \langle x_{i-n}y_i\colon i \in I^k_{m+n}\rangle \colon k\ge 1\rangle/\approx = \Phi_3(x)\cdot \Phi_3(y)/\approx$.

In the similar way we can establish that $\Phi_3$ preserves $\bss$ and $\sslash$, i.e. $\Phi_3$ is a homomorphism. Now it is clear that $\Phi_3$ is injective.
\end{proof}

Now we present the main result of this section.

\begin{theorem}\label{th:6.6}
The variety $\mathsf K$ of integral residuated lattices generated by all kites is generated by all finite-dimensional kites.
\end{theorem}

\begin{proof}
By Theorem \ref{th:5.14}, the variety $\mathsf K$ is generated by all subdirectly irreducible kites. Theorem \ref{th:5.11} describes all infinite-dimensional subdirectly irreducible kites. Up to isomorphism, there are only three non-isomorphic infinite-dimensional subdirectly irreducible kites, and each of them can be embedded into the variety $\mathsf K_f$, the variety generated by all finite-dimensional kites, as it follows from Propositions \ref{pr:6.2}, \ref{pr:6.3}, \ref{pr:6.5}. Therefore, $\mathsf K=\mathsf K_f$.
\end{proof}

\begin{corollary}\label{co:6.7}
The variety $\mathsf K$ is the varietal join of varieties $\mathsf K_n$ of integral residuated lattices generated by $n$-dimensional kites, that is, $\mathsf K=\bigvee_{n=0}^\infty \mathsf K_n$.
\end{corollary}

\section{Homomorphisms between Kites}%7

In the section we show how we can simply construct a homomorphism from one kite $\textstyle K^\kappa_{J_0,J_1}(\mathbf G)$  into another one $\textstyle K^\lambda_{I_0,I_1}(\mathbf G)$.

In the previous sections we presented one construction of a kite which is an integral residuated lattice using an integral residuated lattice and the system of sets $I_1\subseteq I_0$ together with an injection $\lambda\colon I_1\longrightarrow I_0$. We call this system \emph{a frame} in this section and we denote it by $(I_0,I_1,\lambda)$. The main goal is a description of transformations of those frames which correspond (contravariantly) to homomorphisms of residuated lattice.

Our construction is motivated by a well-known construction. Having two sets $I$ and $J$ together with a mapping $f\colon I\longrightarrow J$, then for any algebra $\mathbf A$ of arbitrary type, the mapping
$$
\mathbf A^f\colon A^J\longrightarrow A^I
$$
defined by
$$
A^f(x)(i)=x(f(i)) \mbox{ for all } x\in A \mbox{ and } i\in I,
$$
is a homomorphism. Analogously we define a new concept a ``transformation of frames"
$$t\colon (I_0,I_1,\lambda)\longrightarrow (J_0,J_1,\kappa)
$$
leading to a homomorphism
$$
\mathcal K(t)\colon \textstyle{K}_{J_0,J_1}^\kappa (\mathbf G)\longrightarrow \textstyle{K}_{I_0,I_1}^\lambda (\mathbf G)
$$
for any integral residuated lattice $\mathbf G$.

\begin{definition}\label{DefTrans}
{\rm Let $(I_0,I_1,\lambda)$ and $(J_0,J_1,\kappa)$ be frames. Then the mapping $t\colon I_0\longrightarrow J_0$ is a \emph{transformation of the frames $(I_0,I_1,\lambda)$ and $(J_0,J_1,\kappa)$} if it satisfies:
\begin{itemize}
\item[(1)] $t^{-1}(J_1)=I_1$,
\item[(2)] $t^{-1}\kappa(J_1)=\lambda (I_1)$,
\item[(3)] any $i\in I_1$ satisfies $t\lambda (i)=\kappa t(i)$.

\end{itemize}
}
\end{definition}

To state the main theorem of this section it is necessary to prove several easy lemmas.

\begin{lemma}\label{TranLemma1}
Having a transformation $t$ of the frames $(I_0,I_1,\lambda)$ and $(J_0,J_1,\kappa)$, the equality $t^{-1}(J_n)=I_n$ holds for any $n\in\mathbb N$.
\end{lemma}

\begin{proof}
Firstly, we inductively prove an inclusion $I_n\subseteq t^{-1}(J_n)$. The condition is supposed in Definition \ref{DefTrans}(1) for $n=1$. Let $I_n\subseteq t^{-1}(J_n)$ for some $n\in \mathbb N$. If $i\in I_{n+1}$, then $\lambda (i)\in I_n\subseteq t^{-1}(J_n)$ and consequently $\kappa t(i)=t\lambda (i)\in J_n$. Thus $t(i)\in J_{n+1}$ and $i\in t^{-1}(J_{n+1})$.

Also the converse inclusion $t^{-1}(J_n)\subseteq I_n$ will be proved inductively. The case $n=1$ is clear. If $t^{-1}(J_n)\subseteq I_n$ holds for some $n\in\mathbb N$. Then $i\in t^{-1}(J_{n+1})$ implies $t(i)\in J_{n+1}$ and also $t\lambda (i)=\kappa t (i) \in J_n$. Finally, we obtain $\lambda (i)\in t^{-1}(J_n)\subseteq I_n$ which give us $i\in I_{n+1}$.
\end{proof}

We recall that the injectivity of the mappings $\lambda$ and $\kappa$ guarantees the uniqueness of inverses if it exists.

\begin{lemma}\label{TranLemma2}
Having a transformation $t$ of the frames $(I_0,I_1,\lambda)$ and $(J_0,J_1,\kappa)$, then for any $i\in I_0$, the element $\lambda^{-1}(i)$ exists if and only if $\kappa^{-1}t(i)$ exists, and then $t\lambda^{-1}(i)=\kappa^{-1}t(i)$.
\end{lemma}
\begin{proof}
If $\lambda^{-1}(i)$ exists, using Definition \ref{DefTrans}(3), we obtain $\kappa t\lambda^{-1}(i)=t\lambda\lambda^{-1}(i)=t(i)$ and thus $\kappa^{-1}t(i)$ exists and moreover  $t\lambda^{-1}(i)=\kappa^{-1}t(i)$ holds. Conversely, if $\kappa^{-1}t(i)$ exists then evidently $\kappa^{-1}t(i)\in J_1$ and thus $i\in t^{-1}\kappa (J_1)=\lambda (I_1)$, see Definition \ref{DefTrans}(2). The last proposition yields the existence of $\lambda^{-1}(i)$.
\end{proof}

\begin{lemma}\label{TranLemma3}
Having a frame $(I_0,I_1,\lambda)$ and any $m,n\in\mathbb N$ such that $m\leq n$, then $i\in\lambda^{m+1}(I_{n+1})$ if and only if $i\in I_{n-m}$ and $\lambda^{-1}(i)\in \lambda^m(I_n)$.
\end{lemma}
\begin{proof}
If $i\in \lambda^{m+1}(I_{n+1})\subseteq I_{n-m}$ then $\lambda ^{-(m+1)}(i)$ exists and $\lambda ^{-(m+1)}(i)\in I_{n+1}\subseteq I_n$. Thus $\lambda ^{-1}(i)\in \lambda^m(I_n)$.

Conversely, having $i\in I_{n-m}$ such that $\lambda^{-1}(i)\in \lambda^m(I_n)$, then there exists $\lambda ^{-(m+1)}(i)\in I_n$. The proposition $i\in I_{n-m}$ yields $\lambda ^{-(m+1)}(i) \in I_{n+1}$ and thus $i\in \lambda^{m+1}(I_{n+1})$.
\end{proof}

\begin{lemma}\label{TranLemma4}
Having a transformation $t$ of the frames $(I_0,I_1,\lambda)$ and $(J_0,J_1,\kappa)$, the equality $t^{-1}\kappa^m(J_n)=\lambda^m(I_n)$ holds for any $m,n\in\mathbb N$ such that $m\leq n$.
\end{lemma}

\begin{proof}
The part of this lemma for any $n\in\mathbb N$ and $m=0$ was proved in Lemma \ref{TranLemma1}. The case $1\leq m\leq n$ we prove inductively. It is clear that lemma holds for $n=0,1$.

Let us suppose that $t^{-1}\kappa^m(J_n)=\lambda^m(I_n)$ holds for some $n\in \mathbb N$ and any $m\in\mathbb N$ such that $m\leq n$. If $1\leq m\leq n+1$, using Lemmas \ref{TranLemma1}--\ref{TranLemma3}, we obtain the following equivalencies

\begin{eqnarray*}
&& i\in t^{-1}\kappa^m(J_{n+1})
\\&\Leftrightarrow& t(i)\in \kappa^m(J_{n+1})\\
&\Leftrightarrow& \kappa^{-1}t(i)\in \kappa^{m-1}(J_{n})\mbox{ and }t(i)\in J_{n-m+1}\\
&\Leftrightarrow& t\lambda^{-1}(i)\in \kappa^{m-1}(J_{n})\mbox{ and }i\in t^{-1}(J_{n-m+1})\\
&\Leftrightarrow& \lambda^{-1}(i)\in t^{-1}(\kappa^{m-1}(J_{n}))=\lambda^{m-1}(I_{n})\mbox{ and }i\in I_{n-m+1}\\
&\Leftrightarrow& i\in \lambda^m(I_{n+1}).
\end{eqnarray*}
\end{proof}

We have proved all claims to state the main theorem of the section.

\begin{theorem}\label{th:hom}
Let us have a transformation $t$ of the frames $(I_0,I_1,\lambda)$ and $(J_0,J_1,\kappa)$, and an integral residuated lattice $\mathbf G$. There exists a homomorphism of residuated lattices
$$\mathcal K(t)\colon \textstyle{K}_{J_0,J_1}^\kappa (\mathbf G)\longrightarrow \textstyle{K}_{I_0,I_1}^\lambda (\mathbf G)
$$
defined by
$$\mathcal K(t)(\langle x_i\colon i\in J_n \rangle)=\langle x_{t(i)}\colon i\in I_n \rangle.
$$
\end{theorem}

\begin{proof}
Firstly we prove that the mapping $\mathcal{K}(t)$ preserves supremas and infimas. Let us have $\langle x_i\colon i\in J_n \rangle,\langle y_i\colon i\in J_m \rangle\in \textstyle{K}_{J_0,J_1}^\kappa (\mathbf G)$. If $m\not = n$,  without lost of generality we can assume $m<n$ and thus
\begin{eqnarray*}
\mathcal K(t)(\langle x_i\colon i\in J_n \rangle \vee \langle y_i\colon i\in J_m \rangle) &=& \mathcal K(t)(\langle y_i\colon i\in J_m \rangle)\\
&=&\langle y_{t(i)}\colon i\in I_m \rangle\\
&=&\langle x_{t(i)}\colon i\in I_n\rangle\vee \langle y_{t(i)}\colon i\in I_m \rangle\\
&=&\mathcal K(t)(\langle x_i\colon i\in J_n \rangle) \vee \mathcal K(t)(\langle y_i\colon i\in J_m \rangle)
\end{eqnarray*}
holds. If $m = n$, we obtain
\begin{eqnarray*}
\mathcal K(t)(\langle x_i\colon i\in J_n \rangle \vee \langle y_i\colon i\in J_n \rangle)
&=&\langle (x\vee y)_{t(i)}\colon i\in I_n \rangle\\
&=&\langle x_{t(i)}\colon i\in I_n\rangle\vee \langle y_{t(i)}\colon i\in I_n \rangle\\
&=&\mathcal K(t)(\langle x_i\colon i\in J_n \rangle) \vee \mathcal K(t)(\langle y_i\colon i\in J_n \rangle).
\end{eqnarray*}
Analogously we can prove that the mapping $\mathcal K(t)$ preserves infimas.

To prove that the mapping $\mathcal K(t)$ preserves product, we compute
\begin{eqnarray*}
\mathcal K(t)(\langle x_i\colon i\in J_n \rangle \cdot \langle y_i\colon i\in J_n \rangle)
&=&\mathcal K(t)(\langle x_{\kappa (i)}\cdot y_i \colon i\in J_{n+m} \rangle)  \\
&=&\langle x_{\kappa^m t (i)}\cdot y_{t(i)} \colon i\in I_{n+m} \rangle  \\
&=&\langle x_{t \lambda^m (i)}\cdot y_{t(i)} \colon i\in I_{n+m} \rangle  \\
&=&\langle x_{t(i)}\colon i\in I_n\rangle\cdot \langle y_{t(i)}\colon i\in I_m \rangle\\
&=&\mathcal K(t)(\langle x_i\colon i\in J_n \rangle) \cdot \mathcal K(t)(\langle y_i\colon i\in J_m \rangle).
\end{eqnarray*}
Moreover, preservation of the unit $1$ is clear.

We prove preservation of residuals. Let us have $\langle x_i\colon i\in J_n \rangle,\langle y_i\colon i\in J_m \rangle\in \textstyle{K}_{J_0,J_1}^\kappa (\mathbf G)$ be such that $m\leq n$. Then
\begin{eqnarray*}
\mathcal K(t)(\langle y_i \colon i\in J_m\rangle \ld  \langle x_i \colon i\in J_n\rangle)&=&\mathcal K(t) (\langle z_i\colon i\in J_{n-m}\rangle)\\
&=&\langle z_{t(i)}\colon i\in I_{n-m}\rangle,
\end{eqnarray*}
where
$$
z_j=\left\{\begin{array}{lll}
x_{\kappa^{n-m}(j)}\ld y_j &\mbox{ if }& j\in J_n\\
e &\mbox{ if }& j\not\in J_n\\
\end{array}
\right.
$$
and thus also
$$
z_{t(i)}=\left\{\begin{array}{lll}
x_{\kappa^{n-m}t(i)}\ld y_{t(i)} &\mbox{ if }& t(i)\in J_n\\
e &\mbox{ if }& t(i)\not\in J_n.\\
\end{array}
\right.
$$
On the other hand,
\begin{eqnarray*}
\mathcal K(t)(\langle y_i \colon i\in J_m\rangle) \ld  \mathcal K(t)(\langle x_i \colon i\in J_n\rangle)&=&\langle y_{t(i)} \colon i\in I_m\rangle \ld \langle x_{t(i)} \colon i\in I_n\rangle\\
&=&\langle w_{i}\colon i\in I_{n-m}\rangle,
\end{eqnarray*}
where
$$
w_i=\left\{\begin{array}{lll}
x_{t\lambda^{n-m}(i)}\ld y_{t(i)} &\mbox{ if }& i\in I_n\\
e &\mbox{ if }& i\not\in I_n.\\
\end{array}
\right.
$$
Lemma \ref{TranLemma1} shows that $i\in I_n$ if and only if $t(i)\in J_n$ and consequently $z_{t(i)}=w_i$ for any $i\in I_{n-m}$. We have proved
$$\mathcal K(t)(\langle y_i \colon i\in J_m\rangle \ld  \langle x_i \colon i\in J_n\rangle)=\mathcal K(t)(\langle y_i \colon i\in J_m\rangle) \ld  \mathcal K(t)(\langle x_i \colon i\in J_n\rangle).$$

Analogously to the previous case it satisfies
\begin{eqnarray*}
\mathcal K(t)(\langle x_i \colon i\in J_n\rangle/\langle y_i \colon i\in J_m\rangle)&=&\mathcal K(t) (\langle z_i\colon i\in J_{n-m}\rangle)\\
&=&\langle z_{t(i)}\colon i\in I_{n-m}\rangle,
\end{eqnarray*}
where
$$
z_j=\left\{\begin{array}{lll}
x_{\kappa^{-m}(j)}/ y_{\kappa^{-m}(j)} &\mbox{ if }& j\in \kappa^m(J_n)\\
e &\mbox{ if }& j\not\in \kappa^m(J_n).\\
\end{array}
\right.
$$
and thus also
$$
z_{t(i)}=\left\{\begin{array}{lll}
x_{\kappa^{-m}t(i)}/ y_{\kappa^{-m}t(i)} &\mbox{ if }& t(i)\in \kappa^m(J_n)\\
e &\mbox{ if }& t(i)\not\in \kappa^m(J_n).\\
\end{array}
\right.
$$
On the other hand,
\begin{eqnarray*}
\mathcal K(t)(\langle x_i \colon i\in J_n\rangle) /  \mathcal K(t)(\langle y_i \colon i\in J_m\rangle)&=&\langle x_{t(i)} \colon i\in I_n\rangle / \langle y_{t(i)} \colon i\in I_m\rangle\\
&=&\langle w_{i}\colon i\in I_{n-m}\rangle,
\end{eqnarray*}
where
$$
w_i=\left\{\begin{array}{lll}
x_{t\lambda^{-m}(i)}/ y_{t\lambda^{-m}(i)} &\mbox{ if }& i\in \lambda^m(I_n)\\
e &\mbox{ if }& i\not\in \lambda^m(I_n).\\
\end{array}
\right.
$$
Lemma \ref{TranLemma4} shows that $i\in \lambda^m(I_n)$ if and only if $t(i)\in \kappa^m(J_n)$ and consequently $z_{t(i)}=w_i$ for any $i\in I_{n-m}$. We have proved
$$
\mathcal K(t)(\langle x_i \colon i\in J_n\rangle /  \langle y_i \colon i\in J_m\rangle)=\mathcal K(t)(\langle x_i \colon i\in J_n\rangle) /  \mathcal K(t)(\langle y_i \colon i\in J_m\rangle).
$$
Finally, we have established that $\mathcal K(t)$ is a homomorphism from the kite $\textstyle{K}_{J_0,J_1}^\kappa (\mathbf G)$ into the kite $\textstyle{K}_{I_0,I_1}^\lambda (\mathbf G)$.
\end{proof}

We note that we do not know general conditions to characterize a homomorphism from one kite over $\mathbf G$ into another one over the same $\mathbf G$.

\section{Conclusion}

In the paper we have presented a construction how from an integral residuated lattice $\mathbf G$ and with an injection of one subset into another one we can build up a new integral residuated lattice. The shape of the resulting algebra resembles a Chinese cascade kite, therefore, we call simply this new algebra a kite, see Theorem \ref{th:3.1}. We have presented subdirectly irreducible kites, Theorem \ref{th:5.4}, and we classified finite-dimensional kites by Theorem \ref{th:5.8}, as well as infinitely countable-dimensional kites in Theorem \ref{th:5.11}. We have showed that the variety of integral residuated lattices generated by kites is generated by the class of finite-dimensional kites, see Theorem \ref{th:6.6}. Finally we have showed a simple condition, a frame, which describes a homomorphism from one kite over $\mathbf G$ into another kite over the same $\mathbf G$, Theorem \ref{th:hom}.

The presented paper enriches the class of integral residuated lattices starting from one integral residuated lattice using two sets and an injection from one set into another one.

\end{document}